\documentclass[amstex,12pt]{article}
\usepackage{amsfonts,graphicx,epstopdf}
\usepackage{amssymb,amsmath,amsthm}
\usepackage{txfonts,cite}
\usepackage{color}
\newtheorem{theorem}{Theorem}[section]
\newtheorem{corollary}{Corollary}[section]
\newtheorem{lemma}{Lemma}[section]

\newtheorem{definition}{Definition}[section]
\newtheorem{remark}{Remark}[section]
\newtheorem{proposition}{Proposition}[section]

\newcommand{\beq}{\begin{equation}}
\newcommand{\eeq}{\end{equation}}
\newcommand{\beqn}{\begin{eqnarray}}
\newcommand{\eeqn}{\end{eqnarray}}

\begin{document}

\allowdisplaybreaks

\title{Left fractional Sobolev space via Riemann$-$Liouville derivatives on time scales and its application to a fractional boundary value problem on time scales\thanks{This work is supported by the National Natural Science Foundation of China under Grant No. 11861072.}}
\author{Xing Hu and Yongkun Li\thanks{The corresponding author, Email: yklie@ynu.edu.cn}\\
Department of Mathematics,
Yunnan University\\
Kunming, Yunnan 650091\\
 People's Republic of China}
\date{}
\maketitle{}
\begin{abstract}
 We first prove the equivalence of two definitions of Riemann$-$Liouville fractional integral on time scales, then by the concept of fractional derivative of Riemann$-$Liouville on time scales, we introduce fractional Sobolev spaces, characterize them, define weak fractional derivatives, and show that they coincide with the Riemann$-$Liouville ones on time scales. Next, we prove equivalence of some norms in the introduced spaces and derive their completeness, reflexivity, separability and some imbeddings.  Finally, as an application, by constructing an appropriate variational setting, using the mountain pass theorem and the genus properties, the existence of weak solutions for a class of Kirchhoff-type fractional $p$-Laplacian systems on time scales with boundary condition is studied, and three results of the existence of weak solutions for this problem is obtained.
\end{abstract}

\section{Introduction}
\setcounter{equation}{0}
\indent

To unify the discrete analysis and continuous analysis, and allow a simultaneous treatment of differential and difference equations, Stefan Hilger \cite{t1} proposed the time scale theory and established its related basic theory \cite{t2,t3}. So far, the study of dynamic equations on time scales has attracted worldwide attention.

It is well known that Sobolev space theory is established to study modern differential equation theory and many problems in the field of mathematical analysis. It has become an integral part of analytical mathematics.
In order to study the solvability of boundary value problems of dynamic equations on time scales, Sobolev space theory on time scales is studied in \cite{7,2,2a, XL1}.

On the one hand, in the past few decades, fractional calculus and fractional differential equations have attracted widespread attention in the field of differential equations, as well as in applied mathematics and science. In addition to true mathematical interest and curiosity, this trend is also driven by interesting scientific and engineering applications that have produced fractional differential equation models to better describe (time) memory effects and (space) non-local phenomena \cite{13,14,15,16,17,lj1}. It is the rise of these applications that give new vitality to the field of fractional calculus and fractional differential equations and call for further research in this field.

On the other hand, recently, based on the concept of fractional derivative of Riemann$-$Liouville on time scales \cite{3}, the authors of \cite{XL1} established the fractional Sobolev space on time scales.
However, the authors in the recent work \cite{c1'} pointed out that the definition of fractional integral on time scales proposed in \cite{3} is not the natural one on time scales. And they developed a new notion of Riemann-Liouville fractional integral on time scales, which can well unify the discrete fractional calculus \cite{XL2, XL3} and its continuous counterpart \cite{4}.

Motivated by the above discussion, in order to fix this defect of the fractional Sobolev space on time scales established in \cite{XL1}, in this paper, we want to contribute with the development of this new area on theories of fractional differential equations on time scales.
More precisely, we first show that the concept of Riemann-Liouville fractional integral on time scales from \cite{XL1} coincides with the ones from \cite{6'}, which is significant for us to prove the semigroup properties of Riemann-Liouville fractional integral on time scales. Next, the left fractional Sobolev space in the sense of weak Riemann$-$Liouville derivatives on time scales was constructed via Riemann$-$Liouville derivatives on time scales. Then, as an application of our new theory, we we study the solvability of a class of Kirchhoff-type fractional $p$-Laplacian systems on time scales with boundary condition by using variational methods and the critical point theory.

\section{Preliminaries}
\setcounter{equation}{0}
\indent

In this section, we briefly collect some basic known notations, definitions, and  results that will be used later.

A time scale $\mathbb{T}$
is an arbitrary nonempty closed subset of the real set $\mathbb{R}$ with the topology and ordering inherited from $\mathbb{R}$.
 Throughout this paper, we denote by $\mathbb{T}$ a time scale. We will use the following notations:
$J_{\mathbb{R}}^0=[a,b)$,  $J_{\mathbb{R}}=[a,b]$,
 $J^0=J_{\mathbb{R}}^0\cap{\mathbb{T}}$, $J=J_{\mathbb{R}}\cap{\mathbb{T}}$, $J^k=[a,\rho(b)]\cap\mathbb{T}$.

\begin{definition}\cite{t2}\label{1}
For $t\in\mathbb{T}$ we define the forward jump operator $\sigma:\mathbb{T}\rightarrow\mathbb{T}$ by
$
\sigma(t):=\inf\{s\in\mathbb{T}:s>t\},
$
while the backward jump operator $\rho:\mathbb{T}\rightarrow\mathbb{T}$ is defined by
$
\rho(t):=\sup\{s\in\mathbb{T}:s<t\}.
$
\end{definition}

\begin{remark}\cite{t2}
\begin{itemize}
  \item [$(1)$]
  In Definition \ref{1}, we put $\inf\emptyset=\sup\mathbb{T}$ (i.e., $\sigma(t)=t$ if $\mathbb{T}$ has a maximum $t$) and $\sup\emptyset=\inf\mathbb{T}$ (i.e., $\rho(t)=t$ if $\mathbb{T}$ has a minimum $t$), where $\emptyset$ denotes the empty set.
  \item [$(2)$]
  If $\sigma(t)>t$, we say that $t$ is right$-$scattered, while if $\rho(t)<t$, we say that $t$ is left$-$scattered. Points that are right$-$scattered and left$-$scattered at the same time are called isolated.
  \item [$(3)$]
  If $t<\sup\mathbb{T}$ and $\sigma(t)=t$, we say that $t$ is right$-$dense, while if $t>\inf\mathbb{T}$ and $\rho(t)=t$, we say that $t$ is left$-$dense. Points that are right$-$dense and left$-$dense at the same time are called dense.
  \item[$(4)$]
  The graininess function $\mu:\mathbb{T}\rightarrow[0,\infty)$ is defined by
  $
  \mu(t):=\sigma(t)-t.
  $
  \item [$(5)$]
  The derivative makes use of the set $\mathbb{T}^k$, which is derived from the time scale $\mathbb{T}$ as follows: If $\mathbb{T}$ has a left$-$scattered maximum $M$, then $\mathbb{T}^k:=\mathbb{T}\backslash\{M\}$; otherwise, $\mathbb{T}^k:=\mathbb{T}$.
\end{itemize}
\end{remark}

\begin{definition}\cite{t2}
Assume that $f:\mathbb{T} \rightarrow \mathbb{R}$ is a function and let
$t\in \mathbb{T}^{k}$. Then we define $f^{\Delta}(t)$ to be the
number (provided it exists) with the property that given any
$\varepsilon>0$, there is a neighborhood $U$ of
t (i.e, $U=(t-\delta,t+\delta)\cap \mathbb{T}$ for some
$\delta>0$) such that
\[
\vert f(\sigma(t))-f(s)-f^{\Delta}(t)(\sigma(t)-s)\vert\leq
\varepsilon\vert\sigma(t)-s\vert
\] for all $s\in U$.
We call $f^{\Delta}(t)$ the delta (or Hilger) derivative of f at t.
Moreover, we say that $f$ is delta (or Hilger) differentiable (or in
short: differentiable) on $\mathbb{T}^{k}$ provided $f^{\Delta}(t)$
exists for all $t\in \mathbb{T}^{k}$. The function
$f^{\Delta}:\mathbb{T}^{k} \rightarrow \mathbb{R}$ is then called
the (delta) derivative of $f$ on $\mathbb{T}^{k}$.
\end{definition}

\begin{definition}\cite{t2}
A function $f:\mathbb{T}\rightarrow\mathbb{R}$ is called rd$-$continuous provided it is continuous at right$-$dense points in $\mathbb{T}$ and its left$-$sided limits exist (finite) at left$-$dense points in $\mathbb{T}$. The set of rd$-$continuous functions $f:\mathbb{T}\rightarrow\mathbb{R}$ will be denoted by
$
C_{rd}=C_{rd}(\mathbb{T})=C_{rd}(\mathbb{T},\mathbb{R}).
$
The set of functions $f:\mathbb{T}\rightarrow\mathbb{R}$ that are differentiable and whose derivative is rd$-$continuous is denoted by
$
C_{rd}^1=C_{rd}^1(\mathbb{T})=C_{rd}^1(\mathbb{T},\mathbb{R}).
$
\end{definition}

\begin{theorem}\cite{t3}\label{ft3}
If $a,b\in\mathbb{T}$ and $f,g\in C_{rd}(\mathbb{T})$, then
\begin{eqnarray*}
\int_{J^0}f^\sigma(t)g^\Delta(t)\Delta t=(fg)(b)-(fg)(a)-\int_{J^0}f^\Delta(t)g(t)\Delta t.
\end{eqnarray*}
\end{theorem}

\begin{theorem}\cite{t3}\label{LTm3}
If $f$ is $\Delta-$integrable on $a,b\in\mathbb{T}$, then so is $\vert f\vert$, and
\begin{eqnarray*}
\left\vert\int_a^bf(t)\Delta t\right\vert\leq\int_a^b\vert f(t)\vert\Delta t.
\end{eqnarray*}
\end{theorem}

\begin{definition}\cite{3}
Let $J$ denote a closed bounded interval in $\mathbb{T}$. A function  $F:J\rightarrow\mathbb{R}$ is called a delta antiderivative of function $f:J\rightarrow\mathbb{R}$ provided $F$ is continuous on $J$, delta differentiable at $J$, and $F^\Delta(t)=f(t)$ for all $t\in J$. Then, we define the $\Delta-$integral of $f$ from $a$ to $b$ by
$
\int_a^bf(t)\Delta t:=F(b)-F(a).
$
\end{definition}

\begin{theorem}\cite{L5}\label{LTm4}
The convolution is commutative and associative, that is, for $f,\,g,\,h\in \mathcal{F}$,
\begin{eqnarray*}
f\ast g=g\ast f,\quad (f\ast g)\ast h=f\ast(g\ast h).
\end{eqnarray*}
\end{theorem}

\begin{proposition}\cite{10}\label{2}
$f$ is an increasing continuous function on  $J$. If $F$ is the extension of $f$ to the real interval $J_{\mathbb{R}}$ given by
\begin{equation*}
F(s):=
\left\{
\begin{aligned}
f(s)&,&\quad if\,\,s\in \mathbb{T},\\
f(t)&,&\quad if\,\,s\in (t,\sigma(t))\notin \mathbb{T},
\end{aligned}
\right.
\end{equation*}
then
$\int_a^bf(t)\Delta t\leq\int_a^bF(t)dt.$
\end{proposition}

Motivated by Definition 4 in \cite{c1'} and Definition 2.1 in \cite{L2}, we can present the right Riemann$-$Liouville fractional integral and derivative on time scales as follows:

\begin{definition}\cite{c1'} (Fractional integral on time scales)\label{3} Suppose  $h$ is an integrable function on $J$. Let $0<\alpha\leq1$. Then the left fractional integral of order $\alpha$ of $h$ is defined by
\begin{eqnarray}\label{2.5}
_a^{\mathbb{T}}I_t^\alpha h(t):=\int_a^t\frac{(t-\sigma(s))^{\alpha-1}}{\Gamma(\alpha)}h(s)\Delta s.
\end{eqnarray}
The right fractional integral of order $\alpha$ of $h$ is defined by
\begin{eqnarray*}
_t^{\mathbb{T}}I_b^\alpha h(t):=\int_t^b\frac{(s-\sigma(t))^{\alpha-1}}{\Gamma(\alpha)}h(s)\Delta s,
\end{eqnarray*}
where $\Gamma$ is the gamma function.
\end{definition}

\begin{definition}\cite{c1'} (Riemann$-$Liouville fractional derivative on time scales)\label{4}  Let  $t\in\mathbb{T}$, $0<\alpha\leq1$, and $h:\mathbb{T}\rightarrow\mathbb{R}$. The left Riemann$-$Liouville fractional derivative of order $\alpha$ of $h$ is defined by
\begin{eqnarray}\label{2.6}
_a^{\mathbb{T}}D_t^\alpha h(t):=\bigg({_a^{\mathbb{T}}}I_t^{1-\alpha} h(t)\bigg)^\Delta=\frac{1}{\Gamma(1-\alpha)}\bigg(\int_a^t(t-\sigma(s))^{-\alpha}h(s)\Delta s\bigg)^\Delta.
\end{eqnarray}
The right Riemann$-$Liouville fractional derivative of order $\alpha$ of $h$ is defined by
\begin{eqnarray*}
_t^{\mathbb{T}}D_b^\alpha h(t):=-\bigg({_t^{\mathbb{T}}}I_b^{1-\alpha} h(t)\bigg)^\Delta=\frac{-1}{\Gamma(1-\alpha)}\bigg(\int_t^b(s-\sigma(t))^{-\alpha}h(s)\Delta s\bigg)^\Delta.
\end{eqnarray*}
\end{definition}

Motivated by Definition 4 and Equation (21) in \cite{c1'} and Theorem 2.1 in \cite{L3}, we can present the right Caputo fractional derivative on time scales as follows:

\begin{definition}(Caputo fractional derivative on time scales)\label{cd}  Let  $t\in\mathbb{T}$, $0<\alpha\leq1$  and $h:\mathbb{T}\rightarrow\mathbb{R}$. The left Caputo fractional derivative of order $\alpha$ of $h$ is defined by
\begin{eqnarray*}
{_a^{\mathbb{T}\,C}}D_t^\alpha h(t):={_a^{\mathbb{T}}}I_t^{1-\alpha} h^\Delta(t)=\frac{1}{\Gamma(1-\alpha)}\int_a^t(t-\sigma(s))^{-\alpha}h^\Delta(s)\Delta s.
\end{eqnarray*}
The right Caputo fractional derivative of order $\alpha$ of $h$ is defined by
\begin{eqnarray*}
{_t^{\mathbb{T}\,C}}D_b^\alpha h(t):=-{_t^{\mathbb{T}}}I_b^{1-\alpha} h^\Delta(t)=\frac{-1}{\Gamma(1-\alpha)}\int_t^b(s-\sigma(t))^{-\alpha}h^\Delta(s)\Delta s.
\end{eqnarray*}
\end{definition}

\begin{definition}\cite{3'}\label{d2'}
For $f:\mathbb{T}\rightarrow\mathbb{R}$, the time scale or generalized Laplace transform of $f$, denoted by $\mathcal{L}_{\mathbb{T}}\{f\}$ or $F(z)$, is given by
\begin{eqnarray*}
\mathcal{L}_{\mathbb{T}}\{f\}(z)=F(z):=\int_0^\infty f(t)g^\sigma(t)\Delta t,
\end{eqnarray*}
where $g(t)=e_{\ominus z}(t,0)$.
\end{definition}

\begin{theorem}\cite{3'} (Inversion formula of the Laplace transform)\label{t1'}
Suppose that $F(z)$ is analytic in the region $Re_\mu(z)>Re_\mu(c)$ and $F(z)\rightarrow0$ uniformly as $\vert z\vert\rightarrow\infty$ in this region. Suppose $F(z)$ has finitely many regressive poles of finite order $\{z_1,z_2,\ldots,z_n\}$ and $\widetilde{F}_{\mathbb{R}}(z)$ is the transform of the function $\widetilde{f}(t)$ on $\mathbb{R}$ that corresponds to the transform $F(z)=F_{\mathbb{T}}(z)$ of $f(t)$ on $\mathbb{T}$, If
\begin{eqnarray*}
\int_{c-i\infty}^{c+i\infty}\vert \widetilde{F}_{\mathbb{R}}(z)\vert\vert dz\vert<\infty,
\end{eqnarray*}
then
\begin{eqnarray*}
f(t)=\sum_{i=1}^nRes_{z=z_i}e_z(t,0)F(z),
\end{eqnarray*}
has transform $F(z)$ for all $z$ with $Re(z)>c$.
\end{theorem}

\begin{definition}\cite{6'}\label{d3'} (Riemann$-$Liouville fractional integral on time scales)
Let $\alpha>0$, $\mathbb{T}$ be a time scale, and $f:\mathbb{T}\rightarrow\mathbb{R}$. The left Riemann$-$Liouville fractional integral of $f$ of order $\alpha$ on the time scale $\mathbb{T}$, denoted by $I_{\mathbb{T}}^\alpha f$, is defined by
\begin{eqnarray*}
I_{\mathbb{T}}^\alpha f(t)=\mathcal{L}_{\mathbb{T}}^{-1}\left[\frac{F(z)}{z^\alpha}\right](t).
\end{eqnarray*}
\end{definition}

\begin{theorem}\cite{c1'} (Cauchy Result on Time Scales)\label{t2'}
Let $n\in\{1,2\}$, $\mathbb{T}$ be a time scale with $a,t_1,\ldots,t_n\in\mathbb{T}$, $t_i>a$, $i=1,\ldots,n$, and $f$ an integrable function on $\mathbb{T}$. Then,
\begin{eqnarray*}
\int_{a}^{t_n}\ldots\int_{a}^{t_1} f(t_0)\Delta t_0\ldots\Delta t_{n-1}=\frac{1}{(n-1)!} \int_{a}^{t_n}(t_n-\sigma(s))^{n-1}\Delta s.
\end{eqnarray*}
\end{theorem}

\begin{theorem}\cite{4'} \label{t3'}
If $x:\mathbb{T}\rightarrow\mathbb{C}$ is regulated and $X(t)=\int_0^tx(\tau)\Delta\tau$ for $t\in\mathbb{T}$, then
\begin{eqnarray*}
\mathcal{L}_{\mathbb{T}}\{X\}(z)=\frac{1}{z}\mathcal{L}_{\mathbb{T}}\{x\}(z)
\end{eqnarray*}
for all $z\in\mathcal{D}\{x\}\backslash\{0\}$ such that $\lim\limits_{t\rightarrow\infty}\{X(t)e_{\ominus z}(t)\}=0$.
\end{theorem}

\begin{theorem}\cite{2} \label{8}
A function $f:J\rightarrow\mathbb{R}^N$ is absolutely continuous on $J$ iff $f$ is $\Delta$-differentiable $\Delta-a.e.$ on $J^0$ and
\begin{eqnarray*}
f(t)=f(a)+\int_{[a,t)_\mathbb{T}}f^\Delta(s)\Delta s,\quad \forall t\in J.
\end{eqnarray*}
\end{theorem}

\begin{theorem}\cite{1} \label{9}
A function $f:\mathbb{T}\rightarrow\mathbb{R}$ is absolutely continuous on $\mathbb{T}$ iff the following conditions are satisfied:
\begin{itemize}
  \item [$(i)$]
  $f$ is $\Delta-$differentiable $\Delta-a.e.$ on $J^0$ and $f^\Delta\in L^1(\mathbb{T})$.
  \item [$(ii)$]
  The equality
  \begin{eqnarray*}
  f(t)=f(a)+\int_{[a,t)_\mathbb{T}}f^\Delta(s)\Delta s
  \end{eqnarray*}
holds for every  $t\in\mathbb{T}$.
\end{itemize}
\end{theorem}

\begin{theorem}\cite{5} \label{10}
A function $q:J_{\mathbb{R}}\rightarrow\mathbb{R}^m$ is absolutely continuous iff there exist a constant $c\in\mathbb{R}^m$ and a function $\varphi\in L^1$ such that
\begin{eqnarray*}
q(t)=c+(I_{a^+}^1\varphi)(t),\quad t\in J_{\mathbb{R}}.
\end{eqnarray*}
In this case, we have $q(a)=c$ and $q'(t)=\varphi(t)$, $t\in J_{\mathbb{R}}$ a.e..
\end{theorem}

\begin{theorem}\cite{2} (Integral representation)\label{11}
Let $\alpha\in(0,1)$ and $q\in L^1$. Then, $q$ has a left$-$sided Riemann$-$Liouville derivative $D_{a^+}^\alpha q$ of order $\alpha$ iff there exist a constant $c\in\mathbb{R}^m$ and a function $\varphi\in L^1$ such that
\begin{eqnarray*}
q(t)=\frac{1}{\Gamma(\alpha)}\frac{c}{(t-a)^{1-\alpha}}+(I_{a^+}^\alpha\varphi)(t),\quad t\in J_{\mathbb{R}}\quad a.e..
\end{eqnarray*}
In this case, we have $I_{a^+}^{1-\alpha}q(a)=c$ and $(D_{a^+}^\alpha q)(t)=\varphi(t)$, $t\in J_{\mathbb{R}}$ a.e..
\end{theorem}

\begin{lemma}\cite{7} \label{13}
Let $f\in L_\Delta^1(J^0)$. Then, the following
\begin{eqnarray*}
\int_{J^0}(f\cdot\varphi^\Delta)(s)\Delta s=0,\quad for\,\,every\,\,\varphi\in C_{0,rd}^1(J^k)
\end{eqnarray*}
holds iff there exists a constant $c\in\mathbb{R}$ such that
\begin{eqnarray*}
f\equiv c \quad \Delta-a.e. \,\,on\,\, J^0.
\end{eqnarray*}
\end{lemma}

\begin{definition}\cite{7} \label{14}
Let $p\in \bar{\mathbb{R}}$ be such that $p\geq 1$ and $u:J\rightarrow\bar{\mathbb{R}}$. Say that $u$ belongs to $W_{\Delta}^{1,p}(J)$ iff $u\in L_{\Delta}^p(J^0)$ and there exists $g:J^k\rightarrow\bar{\mathbb{R}}$ such that $g\in L_{\Delta}^p(J^0)$ and
\begin{eqnarray*}
\int_{J^0}(u\cdot\varphi^\Delta)(s)\Delta s=-\int_{J^0}(g\cdot\varphi^\sigma)(s)\Delta s,\quad \forall\varphi\in C_{0,rd}^1(J^k),
\end{eqnarray*}
with
\begin{eqnarray*}
C_{0,rd}^1(J^k):=\bigg\{f:J\rightarrow\mathbb{R}:f\in C_{rd}^1(J^k),\,f(a)=f(b)\bigg\},
\end{eqnarray*}
where $C_{rd}^1(J^k)$ is the set of all continuous functions on $J$ such that they are $\Delta-$differential on $J^k$ and their $\Delta-$derivatives are $rd-$continuous on $J^k$.
\end{definition}

\begin{theorem}\cite{7} \label{15}
Let $p\in\bar{\mathbb{R}}$ be such that $p\geq 1$. Then, the set $L_\Delta^p(J^0)$ is a Banach space together with the norm defined for every $f\in L_\Delta^p(J^0)$ as
\begin{equation*}
\|f\|_{L_\Delta^p}:=
\left\{
\begin{aligned}
&\bigg[\int_{J^0}\vert f\vert^p(s)\Delta s\bigg]^{\frac{1}{p}},&\quad if\,\,p\in \mathbb{R},\\
&\inf\{C\in\mathbb{R}:\vert f\vert\leq C \,\Delta-a.e. \,\, \mathrm{on} \,\,J^0\},&\quad if\,\,p=+\infty.
\end{aligned}
\right.
\end{equation*}
Moreover, $L_\Delta^2(J^0)$ is a Hilbert space together with the inner product given for every $(f,g)\in L_\Delta^2(J^0)\times L_\Delta^2(J^0)$ by
\begin{eqnarray*}
(f,g)_{L_\Delta^2}:=\int_{J^0}f(s)\cdot g(s)\Delta s.
\end{eqnarray*}
\end{theorem}

\begin{theorem}\cite{4} \label{16}
Fractional integration operators are bounded in $L^p(J_{\mathbb{R}})$, i.e., the following estimate
\begin{eqnarray*}
\|I_{a^+}^\alpha\varphi\|_{L^p(a,b)}\leq\frac{(b-a)^{Re\alpha}}{Re\alpha\vert\Gamma(\alpha)\vert}\|\varphi\|_{L^p(J_{\mathbb{R}})},\quad Re\alpha>0
\end{eqnarray*}
holds.
\end{theorem}

\begin{proposition}\cite{7} \label{17}
Suppose $p\in\bar{\mathbb{R}}$ and $p\geq 1$. Let $p'\in\bar{\mathbb{R}}$ be such that $\frac{1}{p'}+\frac{1}{p'}=1$. Then, if $f\in L_\Delta^p(J^0)$ and $g\in L_\Delta^{p'}(J^0)$, then $f\cdot g\in L_\Delta^1(J^0)$ and
\begin{eqnarray*}
\|f\cdot g\|_{L_\Delta^1}\leq\|f\|_{L_\Delta^p}\cdot\|g\|_{L_\Delta^{p'}}.
\end{eqnarray*}
This expression is called H$\ddot{o}$lder's inequality and Cauchy$-$Schwarz's inequality whenever $p=2$.
\end{proposition}

\begin{theorem}\cite{t3} (The first mean value theorem) \label{18}
Let $f$ and $g$ be bounded and integrable functions on $J$, and let $g$ be nonnegative (or nonpositive) on $J$. Let us set
\begin{eqnarray*}
m=\inf\{f(t):t\in J^0\} \quad and \quad M=\sup\{f(t):t\in J^0\}.
\end{eqnarray*}
Then there exists a real number $\Lambda$ satisfying the inequalities $m\leq \Lambda\leq M$ such that
\begin{eqnarray*}
\int_a^bf(t)g(t)\Delta t=\Lambda\int_a^bg(t)\Delta t.
\end{eqnarray*}
\end{theorem}

\begin{corollary}\cite{t3} \label{19}
Let $f$ be an integrable function on $J$ and let $m$ and $M$ be the infimum and supremum, respectively, of $f$ on $J^0$. Then there exists a number $\Lambda$ between $m$ and $M$ such that
$
\int_a^bf(t)\Delta t=\Lambda(b-a).
$
\end{corollary}

\begin{theorem}\cite{t3} \label{20}
Let $f$ be a function defined on $J$ and let $c\in\mathbb{T}$ with $a<c<b$. If $f$ is $\Delta-$integrable from $a$ to $c$ and from $c$ to $b$, then $f$ is $\Delta-$integrable from $a$ to $b$ and
\begin{eqnarray*}
\int_a^bf(t)\Delta t=\int_a^cf(t)\Delta t+\int_c^bf(t)\Delta t.
\end{eqnarray*}
\end{theorem}

\begin{lemma}\cite{t3} \label{OL4}
Assume that $a,b\in\mathbb{T}$. Every constant function $f:\mathbb{T}\rightarrow\mathbb{R}$ is $\Delta-$integrable from $a$ to $b$ and
\begin{eqnarray*}
\int_{J^0}c\Delta t=c(b-a).
\end{eqnarray*}
\end{lemma}

\begin{lemma}\cite{11} (A time scale version of the Arzela$-$Ascoli theorem)\label{21}
Let $X$ be a subset of $C(J,\mathbb{R})$ satisfying the following conditions:
\begin{itemize}
  \item [$(i)$]
  $X$ is bounded;
  \item [$(ii)$]
  For any given $\epsilon>0$, there exists $\delta>0$ such that $t_1,t_2\in J$, $\vert t_1-t_2\vert<\delta$ implies $\vert f(t_1)-f(t_2)\vert<\epsilon$ for all $f\in X$.
\end{itemize}
Then, $X$ is relatively compact.
\end{lemma}

\section{Some fundamental properties of Left Riemann-Liouville fractional operators on time scales}
\setcounter{equation}{0}
\indent

Inspired by \cite{1'}, we can obtain the consistency of Definition \ref{3} and Definition \ref{d3'} by using the above theory of the Laplace transform on time scales and the inverse Laplace transform on time scales.

\begin{theorem}\label{t4'}
Let $\alpha>0$, $\mathbb{T}$ is a time scale, $[a,b]_{\mathbb{T}}$ is an interval of $\mathbb{T}$, and $f$ is an integrable function on $[a,b]_{\mathbb{T}}$. Then, $\left({_a^{\mathbb{T}}I_t^\alpha} f\right)(t)=I_{\mathbb{T}}^\alpha f(t)$.
\end{theorem}

\begin{proof}
Using the Laplace transform on time scale for (\ref{2.5}), in view of Definition \ref{3}, Theorem \ref{t2'}, Theorem \ref{t3'} and Definition \ref{d2'}, we have
{\setlength\arraycolsep{2pt}
\begin{eqnarray}\label{e1'}
&&\mathcal{L}_{\mathbb{T}}\left\{\left({_a^{\mathbb{T}}I_t^\alpha} f\right)(t)\right\}(z)\nonumber\\
&=&\mathcal{L}_{\mathbb{T}}\left\{\frac{1}{\Gamma(\alpha)}
\int_a^t(t-\sigma(s))^{\alpha-1}f(s)\Delta s\right\}(z)\nonumber\\
&=&\mathcal{L}_{\mathbb{T}}\left\{\int_{a}^{t_\alpha}\ldots\int_{a}^{t_1} f(t_0)\Delta t_0\ldots\Delta t_{\alpha-1}\right\}(z)\nonumber\\
&=&\frac{1}{z^\alpha}\mathcal{L}_{\mathbb{T}}\{f\}(z)\nonumber\\
&=&\frac{F(z)}{z^\alpha}(t).
\end{eqnarray}}
Taking the inverse Laplace transform on time scale for (\ref{2.5}), with an eye to Definition \ref{d3'}, one arrives at
\begin{eqnarray*}
\left({_a^{\mathbb{T}}I_t^\alpha} f\right)(t)=\mathcal{L}_{\mathbb{T}}^{-1}\left[\frac{F(z)}{z^\alpha}\right](t)
=I_{\mathbb{T}}^\alpha f(t).
\end{eqnarray*}
The proof is complete.
\end{proof}

Combining with \cite{3, 6'} and Theorem \ref{t4'}, we see that Proposition 15, Proposition 16, Proposition 17, Corollary 18, Theorem 20 and Theorem 21 remain intact under the new Definition \ref{3}.

\begin{proposition}\label{5}
Let $h$ be $\Delta-$integrable on $J$ and $0<\alpha\leq 1$. Then
$_a^{\mathbb{T}}D_t^\alpha h(t)=\Delta\circ{_a^{\mathbb{T}}I_t^{1-\alpha}}h(t).$
\end{proposition}

\begin{proof}
Let $h:\mathbb{T}\rightarrow\mathbb{R}$. In view of $(\ref{2.5})$  and $(\ref{2.6})$, we obtain
\begin{align*}
{_t^\mathbb{T}}D_b^{\alpha}h(t)
=&\frac{1}{\Gamma(1-\alpha)}\bigg(\int_a^t(t-\sigma(s))^{-\alpha}h(s)\Delta s\bigg)^\Delta\\
=&\bigg({_a^{\mathbb{T}}}I_t^{1-\alpha} h(t)\bigg)^\Delta\\
=&\Delta\circ{_a^{\mathbb{T}}I_t^{1-\alpha}}h(t).
\end{align*}
The proof is complete.
\end{proof}

\begin{proposition}\label{6}
For any function $h$ that is integrable on $J$, the Riemann$-$Liouville $\Delta-$fractional integral satisfies
$
_a^{\mathbb{T}}I_t^\alpha\circ{_a^{\mathbb{T}}I_t^\beta}={_a^{\mathbb{T}}}I_t^{\alpha+\beta}
={_a^{\mathbb{T}}I_t^\beta}\circ{_a^{\mathbb{T}}I_t^\alpha}$
for $\alpha>0$ and $\beta>0$.
\end{proposition}

\begin{proof}
Combining with Proposition 3.4 in \cite{6'} and Theorem \ref{t4'}, one gets that
\begin{eqnarray*}
_a^{\mathbb{T}}I_t^\alpha\circ{_a^{\mathbb{T}}I_t^\beta}
={_a^{\mathbb{T}}}I_t^{\alpha+\beta}.
\end{eqnarray*}
In a similarly way, one arrives at
\begin{eqnarray*}
_a^{\mathbb{T}}I_t^\beta\circ{_a^{\mathbb{T}}I_t^\alpha}
={_a^{\mathbb{T}}}I_t^{\alpha+\beta}.
\end{eqnarray*}
Consequently, we obtain that
\begin{eqnarray*}
_a^{\mathbb{T}}I_t^\alpha\circ{_a^{\mathbb{T}}I_t^\beta}={_a^{\mathbb{T}}}I_t^{\alpha+\beta}
={_a^{\mathbb{T}}I_t^\beta}\circ{_a^{\mathbb{T}}I_t^\alpha}.
\end{eqnarray*}
The proof is complete.
\end{proof}

\begin{proposition}\label{173}
For any function $h$ that is integrable on $J$ one has
$
_a^{\mathbb{T}}D_t^\alpha\circ{_a^{\mathbb{T}}I_t^\alpha h}=h.
$
\end{proposition}

\begin{proof}
Taking account of Propositions \ref{5} and \ref{6}, one can get
\begin{eqnarray*}
_a^{\mathbb{T}}D_t^\alpha\circ{_a^{\mathbb{T}}I_t^\alpha h(t)}=\bigg({_a^{\mathbb{T}}}I_t^{1-\alpha} ({_a^{\mathbb{T}}I_t^\alpha} (h(t))\bigg)^\Delta=\left({_a^{\mathbb{T}}}I_th(t)\right)^\Delta=h.
\end{eqnarray*}
The proof is complete.
\end{proof}

\begin{corollary}
For $0<\alpha\leq 1$, we have
$
_a^{\mathbb{T}}D_t^\alpha\circ{_a^{\mathbb{T}}D_t^{-\alpha}}=Id$ and $_a^{\mathbb{T}}I_t^{-\alpha}\circ{_a^{\mathbb{T}}I_t^\alpha}=Id,
$
where $Id$ denotes the identity operator.
\end{corollary}

\begin{proof}
In view of Proposition \ref{173}, we have
\begin{eqnarray*}
{_a^{\mathbb{T}}D_t^\alpha}\circ{_a^{\mathbb{T}}D_t^{-\alpha}}
={_a^{\mathbb{T}}D_t^\alpha}\circ{_a^{\mathbb{T}}I_t^{\alpha}}=Id\quad \mathrm{and} \quad{_a^{\mathbb{T}}I_t^{-\alpha}}\circ{_a^{\mathbb{T}}I_t^\alpha}
={_a^{\mathbb{T}}D_t^{\alpha}}\circ{_a^{\mathbb{T}}I_t^\alpha}=Id.
\end{eqnarray*}
The proof is complete.
\end{proof}

\begin{theorem}\label{thm21}
Let $f\in C(J)$ and $\alpha>0$, then  $f\in {_a^{\mathbb{T}}}I_t^\alpha(J)$ iff
\begin{eqnarray}\label{t3.1}
_a^{\mathbb{T}}I_t^{1-\alpha}f\in C^1(J)
\end{eqnarray}
and
\begin{eqnarray}\label{t3.2}
\bigg({_a^{\mathbb{T}}}I_t^{1-\alpha}f(t)\bigg)\bigg\vert_{t=a}=0,
\end{eqnarray}
where ${_a^{\mathbb{T}}}I_t^\alpha(J)$ denotes the space of functions that can be represented by the left Riemann$-$Liouville $\Delta$-integral of order $\alpha$ of a $C(J)-$function.
\end{theorem}

\begin{proof}
Suppose $f\in{_a^{\mathbb{T}}I_t^\alpha}(J)$, $f(t)={^{\mathbb{T}}_aI_t^\alpha} g(t)$ for some $g\in C(J)$, and
\begin{eqnarray*}
_a^{\mathbb{T}}I_t^{1-\alpha}(f(t))={_a^{\mathbb{T}}I_t^{1-\alpha}}(^{\mathbb{T}}_aI_t^\alpha g(t)).
\end{eqnarray*}
In view of Proposition \ref{6}, one gets
\begin{eqnarray*}
_a^{\mathbb{T}}I_t^{1-\alpha}(f(t))={_a^{\mathbb{T}}I_t}g(t)=\int_a^tg(s)\Delta s.
\end{eqnarray*}
As a result, $_a^{\mathbb{T}}I_t^{1-\alpha}f\in C(J)$ and
\begin{eqnarray*}
\bigg({_a^{\mathbb{T}}}I_t^{1-\alpha}f(t)\bigg)\bigg\vert_{t=a}=\int_a^ag(s)\Delta s=0.
\end{eqnarray*}
Inversely, suppose that $f\in C(J)$ satisfies (\ref{t3.1}) and (\ref{t3.2}). Then, by applying Taylor's formula  to function $_a^{\mathbb{T}}I_t^{1-\alpha}f$, we obtain
\begin{eqnarray*}
_a^{\mathbb{T}}I_t^{1-\alpha}f(t)=\int_a^t\frac{\Delta}{\Delta s}{_a^{\mathbb{T}}I_s^{1-\alpha}}f(s)\Delta s,\quad\forall t\in J.
\end{eqnarray*}
Let $\varphi(t)=\frac{\Delta}{\Delta t}{_a^{\mathbb{T}}I_t^{1-\alpha}}f(t)$. Note that $\varphi\in C(J)$ by (\ref{t3.1}). Now by Proposition \ref{6}, one sees that
\begin{eqnarray*}
_a^{\mathbb{T}}I_t^{1-\alpha}(f(t))={_a^{\mathbb{T}}I_t^1}\varphi(t)
={_a^{\mathbb{T}}I_t^{1-\alpha}}[_a^{\mathbb{T}}I_t^{\alpha}\varphi(t)]
\end{eqnarray*}
and hence
\begin{eqnarray*}
_a^{\mathbb{T}}I_t^{1-\alpha}(f(t))-{_a^{\mathbb{T}}I_t^{1-\alpha}}[_a^{\mathbb{T}}
I_t^{\alpha}\varphi(t)]\equiv0.
\end{eqnarray*}
Therefore, we have
\begin{eqnarray*}
_a^{\mathbb{T}}I_t^{1-\alpha}[f(t)-{_a^{\mathbb{T}}
I_t^{\alpha}}\varphi(t)]\equiv0.
\end{eqnarray*}
From the uniqueness of solution to Abel's integral equation (\cite{L1}), this implies that $f-{_a^{\mathbb{T}}I_t^{\alpha}}\varphi\equiv0$. Hence, $f={_a^{\mathbb{T}}I_t^{\alpha}}\varphi$ and $f\in{_a^{\mathbb{T}}I_t^{\alpha}}(J)$. The proof is complete.
\end{proof}

\begin{theorem}\label{7}
Let $\alpha>0$ and $f\in C(J)$ satisfy the condition in Theorem \ref{thm21}. Then,
\begin{eqnarray*}
(_a^{\mathbb{T}}I_t^\alpha\circ{_a^{\mathbb{T}}D_t^\alpha})(f)=f.
\end{eqnarray*}
\end{theorem}

\begin{proof}
Combining with Theorem \ref{thm21} and Proposition \ref{173}, we can see that
\begin{eqnarray*}
_a^{\mathbb{T}}I_t^{\alpha}\circ{_a^{\mathbb{T}}D_t^{\alpha}}f(t)
={_a^{\mathbb{T}}I_t^{\alpha}}\circ{_a^{\mathbb{T}}D_t^{\alpha}}
(_a^{\mathbb{T}}I_t^{\alpha}\varphi(t))={_a^{\mathbb{T}}I_t^{\alpha}}\varphi(t)=f(t).
\end{eqnarray*}
The proof is complete.
\end{proof}

\begin{theorem}\label{12}
Let $\alpha>0$, $p,q\geq1$, and $\frac{1}{p}+\frac{1}{q}\leq 1+\alpha$, where $p\neq 1$ and $q\neq 1$ in the case when $\frac{1}{p}+\frac{1}{q}=1+\alpha$. Moreover, let
\begin{eqnarray*}
_a^\mathbb{T}I_t^\alpha(L^p):=\bigg\{f:f={_a^\mathbb{T}}I_t^\alpha g,g\in L^p(J)\bigg\}
\end{eqnarray*}
and
\begin{eqnarray*}
_t^\mathbb{T}I_b^\alpha(L^p):=\bigg\{f:f={_t^\mathbb{T}}I_b^\alpha g,g\in L^p(J)\bigg\},
\end{eqnarray*}
then the following integration by parts formulas hold.
\begin{itemize}
  \item [$(a)$]
  If $\varphi\in L^p(J)$ and $\psi\in L^q(J)$, then
  \begin{eqnarray*}
  \int_{J^0}\varphi(t)\bigg({_a^\mathbb{T}}I_t^\alpha\psi\bigg)(t)\Delta t=\int_{J^0}\psi(t)\bigg({_t^\mathbb{T}}I_b^\alpha\varphi\bigg)(t)\Delta t.
  \end{eqnarray*}
  \item [$(b)$]
  If $g\in {_t^\mathbb{T}}I_b^\alpha(L^p)$ and $f\in {_a^\mathbb{T}}I_t^\alpha (L^q)$, then
  \begin{eqnarray*}
  \int_{J^0}g(t)\bigg({_a^\mathbb{T}}D_t^\alpha f\bigg)(t)\Delta t=\int_{J^0}f(t)\bigg({_t^\mathbb{T}}D_b^\alpha g\bigg)(t)\Delta t.
  \end{eqnarray*}
\item [$(c)$]
  For Caputo fractional derivatives, if $g\in {_t^\mathbb{T}}I_b^\alpha(L^p)$ and $f\in {_a^\mathbb{T}}I_t^\alpha (L^q)$, then
  \begin{eqnarray*}
  \int_a^bg(t)\bigg({_a^{\mathbb{T}\,C}}D_t^\alpha f\bigg)(t)\Delta t=\left[{_t^\mathbb{T}}I_b^{1-\alpha}g(t)\cdot f(t)\right]\bigg\vert_{t=a}^b+\int_a^bf(\sigma(t))\bigg({_t^\mathbb{T}}D_b^\alpha g\bigg)(t)\Delta t.
  \end{eqnarray*}
  and
  \begin{eqnarray*}
  \int_a^bg(t)\bigg({_t^{\mathbb{T}\,C}}D_b^\alpha f\bigg)(t)\Delta t=\left[{_a^\mathbb{T}}I_t^{1-\alpha}g(t)\cdot f(t)\right]\bigg\vert_{t=a}^b+\int_a^bf(\sigma(t))\bigg({_a^\mathbb{T}}D_t^\alpha g\bigg)(t)\Delta t.
  \end{eqnarray*}
\end{itemize}
\end{theorem}

\begin{proof}
\begin{itemize}
  \item [$(a)$]
It follows from Definition \ref{3} and Fubini's theorem on time scales that
{\setlength\arraycolsep{2pt}
\begin{eqnarray*}
&&\int_a^b\varphi(t)\bigg({_a^\mathbb{T}}I_t^\alpha\psi\bigg)(t)\Delta t\nonumber\\
&=&\int_a^b\varphi(t)\left(\int_a^t\frac{(t-\sigma(s))^{\alpha-1}}{\Gamma(\alpha)}\psi(s)\Delta s\right)\Delta t\nonumber\\
&=&\int_a^b\psi(s)\int_s^b\frac{(t-\sigma(s))^{\alpha-1}}{\Gamma(\alpha)}\varphi(t)\Delta t\Delta s\nonumber\\
&=&\int_a^b\psi(t)\int_t^b\frac{(s-\sigma(t))^{\alpha-1}}{\Gamma(\alpha)}\varphi(s)\Delta s\Delta t\nonumber\\
&=&\int_a^b\psi(t)\bigg({_t^\mathbb{T}}I_b^\alpha\varphi\bigg)(t)\Delta t.
\end{eqnarray*}}
The proof is complete.
\item [$(b)$]
It follows from Definition \ref{4} and Fubini's theorem on time scales that
{\setlength\arraycolsep{2pt}
\begin{eqnarray*}
&&\int_a^bg(t)\bigg({_a^\mathbb{T}}D_t^\alpha f\bigg)(t)\Delta t\nonumber\\
&=&\int_a^bg(t)\left(\frac{1}{\Gamma(1-\alpha)}\bigg(\int_a^t(t-\sigma(s))^{-\alpha}f(s)\Delta s\bigg)^\Delta\right)\Delta t\nonumber\\
&=&\int_a^bf(s)\left(\frac{1}{\Gamma(1-\alpha)}\bigg(\int_s^b(t-\sigma(s))^{-\alpha}g(t)\Delta t\bigg)^\Delta\right)\Delta s\nonumber\\
&=&\int_a^bf(t)\left(\frac{1}{\Gamma(1-\alpha)}\bigg(\int_t^b(s-\sigma(t))^{-\alpha}g(s)\Delta s\bigg)^\Delta\right)\Delta t\nonumber\\
&=&\int_a^bg(t)\bigg({_t^\mathbb{T}}D_b^\alpha f\bigg)(t)\Delta t.
\end{eqnarray*}}
The proof is complete.
\item [$(c)$]
It follows from Definition \ref{cd}, Fubini's theorem on time scales and Theorem \ref{ft3}  that
{\setlength\arraycolsep{2pt}
\begin{eqnarray*}
&&\int_a^bg(t)\bigg({_a^{\mathbb{T}\,C}}D_t^\alpha f\bigg)(t)\Delta t\nonumber\\
&=&\int_a^bg(t)\left(\frac{1}{\Gamma(1-\alpha)}\int_a^t(t-\sigma(s))^{-\alpha}f^\Delta(s)\Delta s\right)\Delta t\nonumber\\
&=&\int_a^bf^\Delta(s)\left(\frac{1}{\Gamma(1-\alpha)}\int_s^b(t-\sigma(s))^{-\alpha}g(t)\Delta t\right)\Delta s\nonumber\\
&=&\int_a^bf^\Delta(t)\left(\frac{1}{\Gamma(1-\alpha)}\int_t^b(s-\sigma(t))^{-\alpha}g(s)\Delta s\right)\Delta t\nonumber\\
\nonumber\\
&=&\left[{_t^\mathbb{T}}I_b^{1-\alpha}g(t)\cdot f(t)\right]\bigg\vert_{t=a}^b-\int_a^bf(\sigma(t))\left(\frac{1}{\Gamma(1-\alpha)}\int_t^b(s-\sigma(t))^{-\alpha}g(s)\Delta s\right)^\Delta\nonumber\\
&=&\left[{_t^\mathbb{T}}I_b^{1-\alpha}g(t)\cdot f(t)\right]\bigg\vert_{t=a}^b+\int_a^bf(\sigma(t))\left(\frac{-1}{\Gamma(1-\alpha)}\int_t^b(s-\sigma(t))^{-\alpha}g(s)\Delta s\right)^\Delta\nonumber\\
&=&\left[{_t^\mathbb{T}}I_b^{1-\alpha}g(t)\cdot f(t)\right]\bigg\vert_{t=a}^b+\int_a^bf(\sigma(t))\bigg({_t^\mathbb{T}}D_b^\alpha g\bigg)(t)\Delta t.
\end{eqnarray*}}
The second relation is obtained in a similar way. The proof is complete.
\end{itemize}
\end{proof}

\section{Fractional Sobolev spaces on time scales and their properties}
\setcounter{equation}{0}
\indent

In this section, we present and prove some lemmas, propositions and theorems, which are of utmost significance for our main results.

In the following, let $0<a<b$.  Inspired by  Theorems \ref{8}$-$\ref{11}, we  give the following definition.
\begin{definition}
Let $0<\alpha\leq1$. By $AC_{\Delta,a^+}^{\alpha,1}(J,\mathbb{R}^N)$ we denote the set of all functions $f:J\rightarrow\mathbb{R}^N$ that have the representation
\begin{eqnarray}\label{22}
f(t)=\frac{1}{\Gamma(\alpha)}\frac{c}{(t-a)^{1-\alpha}}+\,_a^\mathbb{T}I_t^\alpha\varphi(t),\quad t\in J\quad\Delta-a.e.
\end{eqnarray}
with $c\in\mathbb{R}^N$ and $\varphi\in L_\Delta^1$.
\end{definition}

Then, we have the following result.
\begin{theorem}\label{23}
Let $0<\alpha\leq1$ and $f\in L_\Delta^1$. Then function $f$ has the left Riemann$-$Liouville derivative $_a^\mathbb{T}D_t^\alpha f$ of order $\alpha$ on the interval $J$ iff $f\in AC_{\Delta,a^+}^{\alpha,1}(J,\mathbb{R}^N)$; that is, $f$ has the representation $(\ref{22})$. In such a case,
\begin{eqnarray*}
({_a^\mathbb{T}}I_t^{1-\alpha}f)(a)=c, \quad
({_a^\mathbb{T}}D_t^\alpha f)(t)=\varphi(t),\quad t\in J\quad\Delta-a.e.
\end{eqnarray*}
\end{theorem}

\begin{proof}
Let us assume that $f\in L_\Delta^1$ has a left$-$sided Riemann$-$Liouville derivative $_a^\mathbb{T}D_t^\alpha f$. This means that $_a^\mathbb{T}I_t^{1-\alpha} f$ is (identified to) an absolutely continuous function. From the integral representation of Theorems \ref{8} and \ref{10}, there exist  a constant $c\in\mathbb{R}^N$ and a function $\varphi\in L_\Delta^1$ such that
\begin{eqnarray}\label{24}
({_a^\mathbb{T}}I_t^{1-\alpha}f)(t)=c+\,(_a^\mathbb{T}I_t^1\varphi)(t),\quad t\in J,
\end{eqnarray}
with $({_a^\mathbb{T}}I_t^{1-\alpha}f)(a)=c$ and $\bigg(({_a^\mathbb{T}}I_t^{1-\alpha}f)(t)\bigg)^\Delta=\,{_a^\mathbb{T}}D_t^{\alpha}f(t)=\varphi(t)$, $t\in J\quad\Delta-a.e.$.

By Proposition \ref{6} and applying $_a^\mathbb{T}I_t^\alpha$ to $(\ref{24})$ we obtain
\begin{eqnarray}\label{25}
({_a^\mathbb{T}}I_t^{1}f)(t)=({_a^\mathbb{T}}I_t^{\alpha}c)(t)
+(_a^\mathbb{T}I_t^1{_a^\mathbb{T}}I_t^\alpha\varphi)(t),\quad t\in J\quad\Delta-a.e..
\end{eqnarray}
The result follows from the $\Delta-$differentiability of $(\ref{25})$.

Conversely, let us assume that $(\ref{22})$ holds true. From Proposition \ref{6} and applying ${_a^\mathbb{T}}I_t^{1-\alpha}$ to $(\ref{22})$ we obtain
\begin{eqnarray*}
({_a^\mathbb{T}}I_t^{1-\alpha}f)(t)=c+
({_a^\mathbb{T}}I_t^1 \varphi)(t),\quad t\in J\quad\Delta-a.e.
\end{eqnarray*}
and then, ${_a^\mathbb{T}}I_t^{1-\alpha}f$ has an absolutely continuous representation.  Further, $f$ has a left$-$sided Riemann$-$Liouville derivative ${_a^\mathbb{T}}D_t^{\alpha}f$. This completes the proof.
\end{proof}

\begin{remark}\label{re31}
\begin{itemize}
  \item [$(i)$]
  By $AC_{\Delta,a^+}^{\alpha,p}$ $(1\leq p<\infty)$ we denote the set of all functions $f:J\rightarrow\mathbb{R}^N$ possessing representation $(\ref{22})$ with $c\in\mathbb{R}^N$ and $\varphi\in L_\Delta^p$.
  \item [$(ii)$]
  It is easy to see that  Theorem \ref{23} implies that for any $1\leq p<\infty$, $f$ has the left Riemann$-$Liouville derivative ${_a^\mathbb{T}}D_t^{\alpha}f\in L_\Delta^p$ iff $f\in AC_{\Delta,a^+}^{\alpha,p}$, that is, $f$ has the representation $(\ref{22})$ with $\varphi\in L_\Delta^p$.
\end{itemize}
\end{remark}

\begin{definition}\label{def1}
Let $0<\alpha\leq1$ and let $1\leq p<\infty$. By left Sobolev space of order $\alpha$ we will mean the set $W_{\Delta,a^+}^{\alpha,p}=W_{\Delta,a^+}^{\alpha,p}(J,\mathbb{R}^N)$ given by
\begin{eqnarray*}
W_{\Delta,a^+}^{\alpha,p}:=\bigg\{u\in L_\Delta^p;\,\exists\, g\in L_\Delta^p,\, \forall\varphi\in C_{c,rd}^\infty \,\,\mathrm{such} \,\,\mathrm{that} \int_{J^0}u(t)\cdot{_t^\mathbb{T}}D_b^{\alpha}\varphi(t)\Delta t=\int_{J^0}g(t)\cdot\varphi(t)\Delta t\bigg\}.
\end{eqnarray*}
\end{definition}

\begin{remark}
A function $g$ given in Definition \ref{def1} will be called the weak left fractional derivative of order $0<\alpha\leq1$ of $u$; let us denote it by $^\mathbb{T}u_{a^+}^\alpha$. The uniqueness of this weak derivative follows from \cite{7}.
\end{remark}

We have the following characterization of $W_{\Delta,a^+}^{\alpha,p}$.
\begin{theorem}\label{thm32}
If $0<\alpha\leq1$ and $1\leq p<\infty$, then
$
W_{\Delta,a^+}^{\alpha,p}=AC_{\Delta,a^+}^{\alpha,p}\cap L_\Delta^p.
$
\end{theorem}

\begin{proof}
On the one hand, if $u\in AC_{\Delta,a^+}^{\alpha,p}\cap L_\Delta^p$, then from Theorem \ref{23} it follows that $u$ has   derivative ${_a^\mathbb{T}}D_t^{\alpha}u\in L_\Delta^p$. Theorem \ref{12} implies that
\begin{eqnarray*}
\int_{J^0}u(t)\,{_t^\mathbb{T}}D_b^{\alpha}\varphi(t)\Delta t=\int_{J^0}({_a^\mathbb{T}}D_t^{\alpha}u)(t)\,\varphi(t)\Delta t
\end{eqnarray*}
for any $\varphi\in C_{c,rd}^\infty$. So, $u\in W_{\Delta,a^+}^{\alpha,p}$ with
$
^\mathbb{T}u_{a^+}^\alpha=g={_a^\mathbb{T}}D_t^{\alpha}u\in L_\Delta^p.
$

On the other hand, if $u\in W_{\Delta,a^+}^{\alpha,p}$, then $u\in L_\Delta^p$ and there exists a function $g\in L_\Delta^p$ such that
\begin{eqnarray}\label{26}
\int_{J^0}u(t){_t^\mathbb{T}}D_b^{\alpha}\varphi(t)\Delta t=\int_{J^0}g(t)\varphi(t)\Delta t
\end{eqnarray}
for any $\varphi\in C_{c,rd}^\infty$.
To show that $u\in AC_{\Delta,a^+}^{\alpha,p}\cap L_\Delta^p$ it suffices to check (Theorem \ref{23} and definition of $AC_{\Delta,a^+}^{\alpha,p}$) that $u$ possesses the left Riemann$-$Liouville derivative of order $\alpha$, which belongs to $L_\Delta^p$, that is,   ${_a^\mathbb{T}}I_t^{1-\alpha}u$ is absolutely continuous on $J$ and its delta derivative of $\alpha$ order (existing $\Delta-a.e.$ on $J$) belongs to $L_\Delta^p$.

In fact, let $\varphi\in C_{c,rd}^\infty$, then $\varphi\in {_t^\mathbb{T}}D_b^{\alpha}(C_{rd})$ and ${_t^\mathbb{T}}D_b^{\alpha}\varphi=-({_t^\mathbb{T}}I_b^{1-\alpha})^\Delta$. From Theorem \ref{12}   it follows that
\begin{align}\label{27}
\int_{J^0}u(t){_t^\mathbb{T}}D_b^{\alpha}\varphi(t)\Delta t
=&\int_{J^0}u(t)(-{_t^\mathbb{T}}I_b^{1-\alpha}\varphi)^\Delta(t)\Delta t\nonumber\\
=&\int_{J^0}({_a^\mathbb{T}}D_t^{1-\alpha}{_a^\mathbb{T}}I_t^{1-\alpha}u)(t)(-{_t^\mathbb{T}}I_b^{1-\alpha}\varphi)^\Delta(t)\Delta t\nonumber\\
=&\int_{J^0}({_a^\mathbb{T}}I_t^{1-\alpha}u)(t)(-\varphi)^\Delta(t)\Delta t\nonumber\\
=&-\int_{J^0}({_a^\mathbb{T}}I_t^{1-\alpha}u)(t)\varphi^\Delta(t)\Delta t.
\end{align}
In view of $(\ref{26})$ and $(\ref{27})$, we get
\begin{eqnarray*}
\int_{J^0}({_a^\mathbb{T}}I_t^{1-\alpha}u)(t)\varphi^\Delta(t)\Delta t=-\int_{J^0}g(t)\varphi(t)\Delta t
\end{eqnarray*}
for any $\varphi\in C_{c,rd}^\infty$. So, ${_a^\mathbb{T}}I_t^{1-\alpha}u\in W_{\Delta,a^+}^{1,p}$. Consequently, ${_a^\mathbb{T}}I_t^{1-\alpha}u$ is absolutely continuous and its delta derivative is equal $\Delta-a.e.$ on $[a,b]_\mathbb{T}$ to $g\in L_\Delta^p$. The proof is complete.
\end{proof}

From the  proof of Theorem \ref{thm32} and   the uniqueness of the weak fractional derivative the following theorem follows.
\begin{theorem}\label{28}
If $0<\alpha\leq1$ and $1\leq p<\infty$, then the weak left fractional derivative $^\mathbb{T}u_{a^+}^\alpha$ of a function $u\in W_{\Delta,a^+}^{\alpha,p}$ coincides with its left Riemann$-$Liouville fractional derivative ${_a^\mathbb{T}}D_t^{\alpha}u$ $\Delta-a.e.$ on $J$.
\end{theorem}

\begin{remark}\label{29}
\begin{itemize}
  \item [$(1)$]
  If $0<\alpha\leq1$ and $(1-\alpha)p<1$, then $AC_{\Delta,a^+}^{\alpha,p}\subset L_\Delta^p$ and, consequently,
  \begin{eqnarray*}
  W_{\Delta,a^+}^{\alpha,p}=AC_{\Delta,a^+}^{\alpha,p}\cap L_\Delta^p=AC_{\Delta,a^+}^{\alpha,p}.
  \end{eqnarray*}
  \item [$(2)$]
  If $0<\alpha\leq1$ and $(1-\alpha)p\geq1$, then $W_{\Delta,a^+}^{\alpha,p}=AC_{\Delta,a^+}^{\alpha,p}\cap L_\Delta^p$ is the set of all functions belong to $AC_{\Delta,a^+}^{\alpha,p}$ that satisfy the condition $({_a^\mathbb{T}}I_t^{1-\alpha}f)(a)=0$.
\end{itemize}
\end{remark}

By using the definition of $W_{\Delta,a^+}^{\alpha,p}$ with $0<\alpha\leq1$ and Theorem \ref{28}, one can easily prove the following result.

\begin{theorem}\label{3.4}
Let $0<\alpha\leq1, 1\leq p<\infty$ and $u\in L_\Delta^p$. Then $u\in W_{\Delta,a^+}^{\alpha,p}$ iff there exists a function $g\in L_\Delta^p$ such that
\begin{eqnarray*}
\int_{J^0}u(t){_t^\mathbb{T}}D_b^{\alpha}\varphi(t)\Delta t=\int_{J^0}g(t)\varphi(t)\Delta t,\quad \varphi\in C_{c,rd}^\infty.
\end{eqnarray*}
In such a case, there exists the left Riemann$-$Liouville derivative ${_a^\mathbb{T}}D_t^{\alpha}u$ of $u$ and $g={_a^\mathbb{T}}D_t^{\alpha}u$.
\end{theorem}

\begin{remark}
Function $g$ will be called the weak left fractional derivative of $u\in W_{\Delta,a^+}^{\alpha,p}$ of order $\alpha$. Its uniqueness follows from \cite{7}. From the above theorem it follows that it coincides with an appropriate Riemann$-$Liouville derivative.
\end{remark}

Let us fix $0<\alpha\leq1$ and consider in the space $W_{\Delta,a^+}^{\alpha,p}$ a norm $\|\cdot\|_{W_{\Delta,a^+}^{\alpha,p}}$ given by
\begin{eqnarray*}
\|u\|_{W_{\Delta,a^+}^{\alpha,p}}^p=\|u\|_{L_\Delta^p}^p+\|_a^\mathbb{T}D_t^\alpha u\|_{L_\Delta^p}^p,\quad u\in W_{\Delta,a^+}^{\alpha,p}.
\end{eqnarray*}
(Here $\|\cdot\|_{L_\Delta}^p$ denotes the delta norm in $L_\Delta^p$ (Theorem \ref{15})).

\begin{lemma}\label{30}
Let $0<\alpha\leq1$ and $1\leq p<\infty$. For any $f\in L_\Delta^p([a,b]_{\mathbb{T}},\mathbb{R}^N)$, we have
\begin{eqnarray}\label{LS1}
\|{_a^\mathbb{T}I_\xi^\alpha} f\|_{L_\Delta^p([a,t]_{\mathbb{T}})}\leq \frac{(t-a)^\alpha}{\Gamma(\alpha+1)}\|f\|_{L_\Delta^p([a,t]_{\mathbb{T}})},\quad\mathrm{for}\,\xi\in [a,t]_{\mathbb{T}},\,t\in[a,b]_{\mathbb{T}}.
\end{eqnarray}
That is to say, the fractional integration operator is bounded in $L_\Delta^p$.
\end{lemma}
\begin{proof}
Inspired by  Theorem \ref{16} and the proof of Lemma 3.1 of \cite{L4}, we can prove (\ref{LS1}).

In fact, if $p=1$, in light of Definition \ref{3}, Theorem \ref{LTm3}, Fubini's theorem on time scales and Proposition \ref{2}, we have
{\setlength\arraycolsep{2pt}
\begin{eqnarray}\label{LS2}
&&\|{_a^\mathbb{T}I_\xi^\alpha} f\|_{L_\Delta^1([a,t]_{\mathbb{T}})}\nonumber\\
&=&\int_a^t\vert{_a^\mathbb{T}I_\xi^\alpha} f\vert\Delta \xi\nonumber\\
&=&\frac{1}{\Gamma(\alpha)}\int_a^t\left\vert\int_a^\xi(\xi-\sigma(\tau))
^{\alpha-1}f(\tau)\right\vert\Delta \tau\Delta \xi\nonumber\\
&\leq&\frac{1}{\Gamma(\alpha)}\int_a^t\int_a^\xi(\xi-\sigma(\tau))
^{\alpha-1}\vert f(\tau)\vert\Delta \tau\Delta \xi\nonumber\\
&=&\frac{1}{\Gamma(\alpha)}\int_a^t\vert f(\tau)\vert\Delta \tau\int_\tau^t(\xi-\sigma(\tau))
^{\alpha-1}\Delta \xi\nonumber\\
&\leq&\frac{1}{\Gamma(\alpha)}\int_a^t\vert f(\tau)\vert\Delta \tau\int_\tau^t(\xi-\tau)
^{\alpha-1}d \xi\nonumber\\
&=&\frac{1}{\Gamma(\alpha)}\int_a^t\vert f(\tau)\vert(t-\tau)^\alpha\Delta \tau\nonumber\\
&\leq&\frac{(t-a)^\alpha}{\Gamma(\alpha+1)}\|f\|_{L_\Delta^1([a,t]_{\mathbb{T}})},
\quad\mathrm{for}\,t\in[a,b]_{\mathbb{T}}.
\end{eqnarray}}
Now, suppose that $1<p<\infty$ and $g\in L_\Delta^q([a,b]_{\mathbb{T}},\mathbb{R}^N)$, where $\frac{1}{p}+\frac{1}{q}=1$. In consideration of Theorem \ref{LTm4}, Theorem \ref{LTm3}, Fubini's theorem on time scales, Proposition \ref{17} and Proposition \ref{2}, one arrives at
{\setlength\arraycolsep{2pt}
\begin{eqnarray}\label{LS3}
&&\left\vert\int_a^tg(\xi)\int_a^\xi(\xi-\sigma(\tau))^{\alpha-1}
f(\tau)\Delta\tau\Delta\xi\right\vert\nonumber\\
&=&\left\vert\int_a^tg(\xi)\int_a^\xi\tau^{\alpha-1}
f(\xi-\sigma(\tau))\Delta\tau\Delta\xi\right\vert\nonumber\\
&\leq&\int_a^t\vert g(\xi)\vert\int_a^\xi\tau^{\alpha-1}
\vert f(\xi-\sigma(\tau))\vert\Delta\tau\Delta\xi\nonumber\\
&\leq&\int_a^t\tau^{\alpha-1}\Delta\tau\int_\tau^t\vert g(\xi)\vert
\vert f(\xi-\sigma(\tau))\vert\Delta\xi\nonumber\\
&\leq&\int_a^t\tau^{\alpha-1}\Delta\tau\left(\int_\tau^t\vert g(\xi)\vert^q\Delta\xi\right)^{\frac{1}{q}}\left(\int_\tau^t\vert f(\xi-\sigma(\tau))\vert^p\Delta\xi\right)^{\frac{1}{p}}\nonumber\\
&\leq&\int_a^t\tau^{\alpha-1}d\tau\|g\|_{L_\Delta^q([a,t]_{\mathbb{T}})}
\|f\|_{L_\Delta^p([a,t]_{\mathbb{T}})}\nonumber\\
&=&\frac{(t-a)^\alpha}{\alpha}\|f\|_{L_\Delta^p([a,t]_{\mathbb{T}})}
\|g\|_{L_\Delta^q([a,t]_{\mathbb{T}})},\quad\mathrm{for}\,t\in[a,b]_{\mathbb{T}}.
\end{eqnarray}}
For any fixed $t\in[a,b]_{\mathbb{T}}$, consider the functional $H_{\xi\ast f}:L_\Delta^q([a,b]_{\mathbb{T}},\mathbb{R}^N)\rightarrow\mathbb{R}$

\begin{align}\label{LS4}
H_{\xi\ast f}(g)=\int_a^t\left[\int_a^\xi(\xi-\sigma(\tau))^{\alpha-1}f(\tau)\Delta\tau\right] g(\xi)\Delta\xi
\end{align}
According to (\ref{LS3}), it is obvious that $H_{\xi\ast f}\in\left(L_\Delta^q([a,b]_{\mathbb{T}},\mathbb{R}^N)\right)^*$, where $\left(L_\Delta^q([a,b]_{\mathbb{T}},\mathbb{R}^N)\right)^*$ denotes the dual space of $L_\Delta^q([a,b]_{\mathbb{T}},\mathbb{R}^N)$. Therefore, by (\ref{LS3}) and (\ref{LS4}) and the Riesz representation theorem, there exists $h\in L_\Delta^p([a,b]_{\mathbb{T}},\mathbb{R}^N)$ such that
\begin{align}\label{LS5}
\int_a^th(\xi)g(\xi)\Delta\xi=\int_a^t\left[\int_a^\xi(\xi-\sigma(\tau))^{\alpha-1}f(\tau)\Delta\tau\right] g(\xi)\Delta\xi
\end{align}
and
\begin{align}\label{LS6}
\|h\|_{L_\Delta^p([a,t]_{\mathbb{T}})}
=\|H_{\xi\ast f}\|_{L_\Delta^p([a,t]_{\mathbb{T}})}
\leq\frac{(t-a)^\alpha}{\alpha}\|f\|_{L_\Delta^p([a,t]_{\mathbb{T}})}
\end{align}
for all $g\in L_\Delta^q([a,b]_{\mathbb{T}},\mathbb{R}^N)$. Hence, we have by (\ref{LS5}) and Definition \ref{3}
\begin{align*}
\frac{1}{\Gamma(\alpha)}h(\xi)
=\frac{1}{\Gamma(\alpha)}\int_a^\xi(\xi-\sigma(\tau))^{\alpha-1}f(\tau)\Delta\tau
={_a^\mathbb{T}I_\xi^\alpha} f(\xi),\quad\mathrm{for}\,\xi\in[a,t]_{\mathbb{T}},
\end{align*}
which means that
\begin{align}\label{LS7}
\|{_a^\mathbb{T}I_\xi^\alpha} f\|_{L_\Delta^p([a,t]_{\mathbb{T}})}
=\frac{1}{\Gamma(\alpha)}\|h\|_{L_\Delta^p([a,t]_{\mathbb{T}})}
\leq\frac{(t-a)^\alpha}{\Gamma(\alpha+1)}\|f\|_{L_\Delta^p([a,t]_{\mathbb{T}})}
\end{align}
according to (\ref{LS6}). Combining with (\ref{LS2}) and (\ref{LS7}), we obtain inequality (\ref{LS1}). The proof is complete.
\end{proof}

\begin{theorem}\label{thm35}
If $0<\alpha\leq1$, then the norm $\|\cdot\|_{W_{\Delta,a^+}^{\alpha,p}}$ is equivalent to the norm $\|\cdot\|_{a, W_{\Delta,a^+}^{\alpha,p}}$ given by
\begin{eqnarray*}
\|u\|_{a,W_{\Delta,a^+}^{\alpha,p}}^p=\vert{_a^\mathbb{T}I_t^{1-\alpha}}u(a)\vert^p+\|_a^\mathbb{T}D_t^\alpha u\|_{L_\Delta^p}^p,\quad u\in W_{\Delta,a^+}^{\alpha,p}.
\end{eqnarray*}
\end{theorem}
\begin{proof}
  $(1)$
  Assume that $(1-\alpha)p<1$. On the one hand, in view of Remarks \ref{re31} and \ref{29}, for $u\in W_{\Delta,a^+}^{\alpha,p}$, we can write it as
  \begin{eqnarray*}
  u(t)=\frac{1}{\Gamma(\alpha)}\frac{c}{(t-a)^{1-\alpha}}+{_a^\mathbb{T}I_t^\alpha\varphi(t)}
  \end{eqnarray*}
  with $c\in\mathbb{R}^N$ and $\varphi\in L_\Delta^p$. Since $(t-a)^{(\alpha-1)p}$ is an increasing monotone function, by using Proposition \ref{2}, we can write that $\int_{J^0}(t-a)^{(\alpha-1)p}\Delta t\leq \int_{J_{\mathbb{R}}^0}(t-a)^{(\alpha-1)p}dt$. And taking into account Lemma \ref{30}, we have
  \begin{align*}
  \|u\|_{L_\Delta^p}^p
  =&\int_{J^0}\bigg\vert\frac{1}{\Gamma(\alpha)}\frac{c}{(t-a)^{1-\alpha}}
  +_a^\mathbb{T}I_t^\alpha\varphi(t)\bigg\vert^p\Delta t\\
  \leq&2^{p-1}\bigg(\frac{\vert c\vert^p}{\Gamma^p(\alpha)}\int_{J^0}(t-a)^{(\alpha-1)p}\Delta t+\|_a^\mathbb{T}I_t^\alpha\varphi\|_{L_\Delta^p}^p\bigg)\\
  \leq&2^{p-1}\bigg(\frac{\vert c\vert^p}{\Gamma^p(\alpha)}\int_{J_{\mathbb{R}}^0}(t-a)^{(\alpha-1)p}dt +\|_a^\mathbb{T}I_t^\alpha\varphi\|_{L_\Delta^p}^p\bigg)\\
  \leq&2^{p-1}\bigg(\frac{\vert c\vert^p}{\Gamma^p(\alpha)}{1}{(\alpha-1)p+1}(b-a)^{(\alpha-1)p+1} +K^p\|\varphi\|_{L_\Delta^p}^p\bigg),
  \end{align*}
  where $K=\frac{(b-a)^\alpha}{\Gamma(\alpha+1)}$. Noting that $c={_a^\mathbb{T}}I_t^{1-\alpha}u(a)$, $\varphi={_a^\mathbb{T}}D_t^\alpha u$,   one can obtain
  \begin{align*}
  \|u\|_{L_\Delta^p}^p
  \leq&L_{\alpha,0}(\vert c\vert^p+\|\varphi\|_{L_\Delta^p}^p)\\
  \leq&L_{\alpha,0}\bigg(\vert{_a^\mathbb{T}}I_t^{1-\alpha}u(a)\vert^p+\|{_a^\mathbb{T}}D_t^\alpha u\|_{L_\Delta^p}^p\bigg)\\
  =&L_{\alpha,0}\|u\|_{a,W_{\Delta,a^+}^{\alpha,p}}^p,
  \end{align*}
  where
  \begin{eqnarray*}
  L_{\alpha,0}=2^{p-1}\bigg(\frac{(b-a)^{1-(1-\alpha)p}}{\Gamma^p(\alpha)(1-(1-\alpha)p)} +K^p\bigg).
  \end{eqnarray*}
  Consequently,
  \begin{align*}
  \|u\|_{W_{\Delta,a^+}^{\alpha,p}}^p
  =&\|u\|_{L_\Delta^P}^P+\|_a^\mathbb{T}D_t^\alpha u\|_{L_\Delta^p}^p\\
  \leq&L_{\alpha,1}\|u\|_{a,W_{\Delta,a^+}^{\alpha,p}}^p,
  \end{align*}
  where $L_{\alpha,1}=L_{\alpha,0}+1$.

  On the other hand,  we will prove that there exists a constant $M_{\alpha,1}$ such that
  \begin{eqnarray}\label{st}
  \|u\|_{a,W_{\Delta,a^+}^{\alpha,p}}^p\leq M_{\alpha,1}\|u\|_{W_{\Delta,a^+}^{\alpha,p}}^p,
  \quad u\in W_{\Delta,a^+}^{\alpha,p}.
  \end{eqnarray}
  Indeed, let $u\in W_{\Delta,a^+}^{\alpha,p}$ and consider  coordinate functions $({_a^\mathbb{T}}I_t^{1-\alpha}u)^i$ of ${_a^\mathbb{T}}I_t^{1-\alpha}u$ with   $i\in\{1,\ldots,N\}$. Lemma \ref{30}, Theorem \ref{18} and Corollary \ref{19} imply that there exist constants
  \begin{eqnarray*}
  \Lambda_i\in
  \bigg[\inf_{t\in[a,b)_\mathbb{T}}({_a^\mathbb{T}}I_t^{1-\alpha}u)^i(t),
  \sup_{t\in[a,b)_\mathbb{T}}({_a^\mathbb{T}}I_t^{1-\alpha}u)^i(t)\bigg],\quad (i=1,2,\ldots,N)
  \end{eqnarray*}
  such that
  \begin{eqnarray*}
  \Lambda_i=\frac{1}{b-a}\int_a^b
  ({_a^\mathbb{T}}I_t^{1-\alpha}u)^i(s)\Delta s.
  \end{eqnarray*}
  Hence, for a fixed $t_0\in J^0$, if $({_a^\mathbb{T}}I_t^{1-\alpha}u)^i(t_0)\neq0$ for all $i=1,2,\ldots,N$, then we can take constants $\theta_i$ such that

  \begin{eqnarray*}
  \theta_i({_a^\mathbb{T}}I_t^{1-\alpha}u)^i(t_0)=\Lambda_i=\frac{1}{b-a}\int_a^b
  ({_a^\mathbb{T}}I_t^{1-\alpha}u)^i(s)\Delta s.
  \end{eqnarray*}
  Therefore, we have
  \begin{eqnarray*}
  ({_a^\mathbb{T}}I_t^{1-\alpha}u)^i(t_0)=\frac{\theta_i}{b-a}\int_a^b
  ({_a^\mathbb{T}}I_t^{1-\alpha}u)^i(s)\Delta s.
  \end{eqnarray*}
  From the absolute continuity (Theorem \ref{9}) of $({_a^\mathbb{T}}I_t^{1-\alpha}u)^i$ it follows that
  \begin{eqnarray*}
  ({_a^\mathbb{T}}I_t^{1-\alpha}u)^i(t)=({_a^\mathbb{T}}I_t^{1-\alpha}u)^i(t_0)
  +\int_{[t_0,t)_\mathbb{T}}\bigg[({_a^\mathbb{T}}I_t^{1-\alpha}u)^i(s)\bigg]^\Delta \Delta s
  \end{eqnarray*}
  for any $t\in J$. Consequently, combining with Proposition \ref{5} and Lemma \ref{30}, we see that
  \begin{align*}
  \vert({_a^\mathbb{T}}I_t^{1-\alpha}u)^i(t)\vert
  =&\bigg\vert({_a^\mathbb{T}}I_t^{1-\alpha}u)^i(t_0)
  +\int_{[t_0,t)_\mathbb{T}}\bigg[({_a^\mathbb{T}}I_t^{1-\alpha}u)^i(s)\bigg]^\Delta \Delta s\bigg\vert\\
  \leq&\frac{\vert\theta_i\vert}{b-a}
  \|{_a^\mathbb{T}}I_t^{1-\alpha}u\|_{L_\Delta^1}+\int_{[t_0,t)_\mathbb{T}}
  \vert({_a^\mathbb{T}}D_t^{\alpha}u)(s)\vert\Delta s\\
  \leq&\frac{\vert\theta_i\vert}{b-a}
  \|{_a^\mathbb{T}}I_t^{1-\alpha}u\|_{L_\Delta^1}
  +\|{_a^\mathbb{T}}D_t^{\alpha}u\|_{L_\Delta^1}\\
  \leq&\frac{\vert\theta_i\vert}{b-a}\frac{(b-a)^{1-\alpha}}{\Gamma(2-\alpha)}
  \|u\|_{L_\Delta^1}+\|{_a^\mathbb{T}}D_t^{\alpha}u\|_{L_\Delta^1}\\
  \end{align*}
  for $t\in J$. In particular,
  \begin{eqnarray*}
  \vert({_a^\mathbb{T}}I_t^{1-\alpha}u)^i(a)\vert\leq\frac{\vert\theta_i\vert}{b-a}\frac{(b-a)^{1-\alpha}}
  {\Gamma(2-\alpha)}\|u\|_{L_\Delta^1}+\|{_a^\mathbb{T}}D_t^{\alpha}u\|_{L_\Delta^1}.
  \end{eqnarray*}
  So,
  \begin{align*}
  \vert({_a^\mathbb{T}}I_t^{1-\alpha}u)(a)\vert
  \leq&N\bigg(\frac{\vert\theta\vert(b-a)^{-\alpha}}
  {\Gamma(2-\alpha)}+1\bigg)\left(\|u\|_{L_\Delta^1}
  +\|{_a^\mathbb{T}}D_t^{\alpha}u\|_{L_\Delta^1}\right)\\
  \leq& NM_{\alpha,0}(b-a)^{\frac{p-1}{p}}\left(\|u\|_{L_\Delta^p}
  +\|{_a^\mathbb{T}}D_t^{\alpha}u\|_{L_\Delta^p}\right),
  \end{align*}
  where $\vert\theta\vert=\max\limits_{i\in\{1,2,\ldots,N\}}\vert\theta_i\vert$ and $M_{\alpha,0}=\frac{\vert\theta\vert(b-a)^{-\alpha}}{\Gamma(2-\alpha)}+1$. Thus,
  \begin{align*}
  \vert({_a^\mathbb{T}}I_t^{1-\alpha}u)(a)\vert^p
  \leq&N^pM^p_{\alpha,0}(b-a)^{p-1}2^{p-1}\left(\|u\|^p_{L_\Delta^p}
  +\|{_a^\mathbb{T}}D_t^{\alpha}u\|^p_{L_\Delta^p}\right),
  \end{align*}
  and, consequently,
  \begin{align*}
  \|u\|_{a,W_{\Delta,a^+}^{\alpha,p}}^p
  =&\vert{_a^\mathbb{T}I_t^{1-\alpha}}u(a)\vert^p+\|_a^\mathbb{T}D_t^\alpha u\|_{L_\Delta^p}^p\\
  \leq&\bigg(N^pM^p_{\alpha,0}(b-a)^{p-1}2^{p-1}+1\bigg)\left(\|u\|^p_{L_\Delta^p}
  +\|{_a^\mathbb{T}}D_t^{\alpha}u\|^p_{L_\Delta^p}\right)\\
  =&M_{\alpha,1}\|u\|_{W_{\Delta,a^+}^{\alpha,p}}^p,
  \end{align*}
  where $M_{\alpha,1}=N^pM^p_{\alpha,0}(b-a)^{p-1}2^{p-1}+1$.

  If $({_a^\mathbb{T}}I_t^{1-\alpha}u)^i(t_0)=0$ for $i$ belongs to some subset of   $\{1,2,\ldots,N\}$,
  from the above argument process one can easily see that there exists a constant $M_{\alpha,1}$ such that \eqref{st} holds.

  $(2)$
  When $(1-\alpha)p\geq1$, then (Remark \ref{29}) $W_{\Delta,a^+}^{\alpha,p}=AC_{\Delta,a^+}^{\alpha,p}\cap L_\Delta^p$ is the set of all functions belong to $AC_{\Delta,a^+}^{\alpha,p}$ that satisfy the condition $({_a^\mathbb{T}}I_t^{1-\alpha}u)(a)=0$. Hence, in the same way as in the case of $(1-\alpha)p<1$ (putting $c=0$), we obtain the inequality
  \begin{eqnarray*}
  \|u\|_{W_{\Delta,a^+}^{\alpha,p}}^p\leq L_{\alpha,1}\|u\|_{a,W_{\Delta,a^+}^{\alpha,p}}^p,\quad \mathrm{with}\,\,\mathrm{some}\,\,L_{\alpha,1}>0.
  \end{eqnarray*}
  The inequality
  \begin{eqnarray*}
  \|u\|_{a,W_{\Delta,a^+}^{\alpha,p}}^p\leq M_{\alpha,1}\|u\|_{W_{\Delta,a^+}^{\alpha,p}}^p,\quad \mathrm{with}\,\,\mathrm{some}\,\,M_{\alpha,1}>0
  \end{eqnarray*}
  is obvious (it is sufficient to put $M_{\alpha,1}=1$ and use the fact that $({_a^\mathbb{T}}I_t^{1-\alpha}u)(a)=0$).

The proof is complete.
\end{proof}

Now, we are in a position to prove some basic properties of the  space $W_{\Delta,a^+}^{\alpha,p}$.
\begin{theorem}\label{3.6}
The space $W_{\Delta,a^+}^{\alpha,p}$ is complete with respect to each of the norms $\|\cdot\|_{W_{\Delta,a^+}^{\alpha,p}}$ and $\|\cdot\|_{a, W_{\Delta,a^+}^{\alpha,p}}$ for any $0<\alpha\leq1$, $1\leq p<\infty$.
\end{theorem}

\begin{proof}In view of Theorem \ref{thm35}, we only need to  show that $W_{\Delta,a^+}^{\alpha,p}$ with the norm $\|\cdot\|_{a, W_{\Delta,a^+}^{\alpha,p}}$ is complete. Let $\{u_k\}\subset W_{\Delta,a^+}^{\alpha,p}$ be a Cauchy sequence with respect to this norm. So, the sequences $\{{_a^\mathbb{T}}I_t^{1-\alpha}u_k(a)\}$  and $\{_a^\mathbb{T}D_t^\alpha u_k\}$ are   Cauchy sequences in   $\mathbb{R}^N$ and $L_\Delta^p$, respectively.

Let $c\in\mathbb{R}^N$ and $\varphi\in L_\Delta^p$ be the limits of the above two sequences in $\mathbb{R}^N$ and $L_\Delta^p$, respectively. Then the function
\begin{eqnarray*}
u(t)=\frac{c}{\Gamma(\alpha)}{(t-a)^{\alpha-1}}+{_a^\mathbb{T}}I_t^\alpha\varphi(t),\quad t\in J\quad\Delta-a.e.
\end{eqnarray*}
belongs to $W_{\Delta,a^+}^{\alpha,p}$ and it is the limit of $\{u_k\}$ in $W_{\Delta,a^+}^{\alpha,p}$ with respect to $\|\cdot\|_{a, W_{\Delta,a^+}^{\alpha,p}}$. The proof is complete.
\end{proof}

The proof method of the following two theorems is inspired by the method used in the proof of Proposition 8.1 $(b), (c)$ in \cite{9}.
\begin{theorem}\label{31}
The space $W_{\Delta,a^+}^{\alpha,p}$ is reflexive with respect to the norm $\|\cdot\|_{W_{\Delta,a^+}^{\alpha,p}}$ for any $0<\alpha\leq1$ and $1<p<\infty$.
\end{theorem}

\begin{proof}
Let us consider $W_{\Delta,a^+}^{\alpha,p}$ with the norm $\|\cdot\|_{W_{\Delta,a^+}^{\alpha,p}}$ and define a mapping
\begin{eqnarray*}
\lambda:W_{\Delta,a^+}^{\alpha,p}\ni u\mapsto\left(u,\, _a^\mathbb{T}D_t^\alpha u\right)\in L_\Delta^p\times L_\Delta^p.
\end{eqnarray*}
It is obvious that
\begin{eqnarray*}
\|u\|_{W_{\Delta,a^+}^{\alpha,p}}=\|\lambda u\|_{L_\Delta^p\times L_\Delta^p},
\end{eqnarray*}
where
\begin{eqnarray*}
\|\lambda u\|_{L_\Delta^p\times L_\Delta^p}=\bigg(\sum_{i=1}^2\|(\lambda u)_i\|_{L_\Delta^p}^p\bigg)^{\frac{1}{p}},\quad \lambda u=\left(u,\, _a^\mathbb{T}D_t^\alpha u\right)\in L_\Delta^p\times L_\Delta^p,
\end{eqnarray*}
which means that the operator $\lambda:u\mapsto\left(u,\, _a^\mathbb{T}D_t^\alpha u\right)$ is a isometric isomorphic mapping and the space $W_{\Delta,a^+}^{\alpha,p}$ is isometric isomorphic to the space $\Omega=\bigg\{\left(u,\, _a^\mathbb{T}D_t^\alpha u\right):\forall u\in W_{\Delta,a^+}^{\alpha,p}\bigg\}$, which is a closed subset of $L_\Delta^p\times L_\Delta^p$ as $W_{\Delta,a^+}^{\alpha,p}$ is closed.

Since $L_\Delta^p$ is reflexive, the Cartesian product space $L_\Delta^p\times L_\Delta^p$ is also a reflexive space with respect to the norm $\|v\|_{L_\Delta^p\times L_\Delta^p}=\bigg(\sum\limits_{i=1}^2\|v_i\|_{L_\Delta^p}^p\bigg)^{\frac{1}{p}}$, where $v=\left(v_1,\, v_2\right)\in L_\Delta^p\times L_\Delta^p$.

Thus, $W_{\Delta,a^+}^{\alpha,p}$ is reflexive with respect to the norm $\|\cdot\|_{W_{\Delta,a^+}^{\alpha,p}}$. The proof is complete.
\end{proof}

\begin{theorem}
The space $W_{\Delta,a^+}^{\alpha,p}$ is separable with respect to the norm $\|\cdot\|_{W_{\Delta,a^+}^{\alpha,p}}$ for any $0<\alpha\leq1$ and $1\leq p<\infty$.
\end{theorem}

\begin{proof}
Let us consider $W_{\Delta,a^+}^{\alpha,p}$ with the norm $\|\cdot\|_{W_{\Delta,a^+}^{\alpha,p}}$ and the mapping $\lambda$ defined in the proof of Theorem \ref{31}. Obviously, $\lambda(W_{\Delta,a^+}^{\alpha,p})$ is separable as a subset of separable space $L_\Delta^p\times L_\Delta^p$. Since $\lambda$ is the isometry, $W_{\Delta,a^+}^{\alpha,p}$ is also separable with respect to the norm $\|\cdot\|_{W_{\Delta,a^+}^{\alpha,p}}$. The proof is complete.
\end{proof}

\begin{proposition}\label{034}
Let $0<\alpha\leq1$ and $1<p<\infty$. For all $u\in W_{\Delta,a^+}^{\alpha,p}$, if $1-\alpha\geq\frac{1}{p}$ or $\alpha>\frac{1}{p}$, then
\begin{eqnarray}\label{33}
\|u\|_{L_\Delta^p}\leq\frac{b^\alpha}{\Gamma(\alpha+1)}\left\|_a^\mathbb{T}D_t^\alpha u\right\|_{L_\Delta^p};
\end{eqnarray}
if $\alpha>\frac{1}{p}$ and $\frac{1}{p}+\frac{1}{q}=1$, then
\begin{eqnarray}\label{34}
\|u\|_\infty\leq\frac{b^{\alpha-\frac{1}{p}}}{\Gamma(\alpha)((\alpha-1)q+1)^{\frac{1}{q}}}
\left\|_a^\mathbb{T}D_t^\alpha u\right\|_{L_\Delta^p}.
\end{eqnarray}
\end{proposition}

\begin{proof}
In view of Remark \ref{29} and Theorem \ref{7}, in order to prove inequalities $(\ref{33})$ and $(\ref{34})$, we only need to prove that
\begin{eqnarray}\label{35}
\left\|_a^\mathbb{T}I_t^\alpha(_a^\mathbb{T}D_t^\alpha u) \right\|_{L_\Delta^p}\leq\frac{b^\alpha}{\Gamma(\alpha+1)}\left\|_a^\mathbb{T}D_t^\alpha u\right\|_{L_\Delta^p}
\end{eqnarray}
for $1-\alpha\geq\frac{1}{p}$ or $\alpha>\frac{1}{p}$, and that
\begin{eqnarray}\label{36}
\left\|_a^\mathbb{T}I_t^\alpha(_a^\mathbb{T}D_t^\alpha u) \right\|_\infty\leq\frac{b^{\alpha-\frac{1}{p}}}{\Gamma(\alpha)((\alpha-1)q+1)
^{\frac{1}{q}}}\left\|_a^\mathbb{T}D_t^\alpha u\right\|_{L_\Delta^p}
\end{eqnarray}
for $\alpha>\frac{1}{p}$ and $\frac{1}{p}+\frac{1}{q}=1$.

Note that $_a^\mathbb{T}D_t^\alpha u\in L_\Delta^p(J, \mathbb{R}^N)$, the inequality $(\ref{35})$ follows from Lemma \ref{30} directly.

We are now in a position to prove $(\ref{36})$. For $\alpha>\frac{1}{p}$, choose $q$ such that $\frac{1}{p}+\frac{1}{q}=1$. For all $u\in W_{\Delta,a^+}^{\alpha,p}$, since $(t-s)^{(\alpha-1)q}$ is an increasing monotone function, by using Proposition \ref{2}, we find that $\int_a^t(t-\sigma(s))^{(\alpha-1)q}\Delta s\leq \int_a^t(t-s)^{(\alpha-1)q}ds$. Taking into account of Proposition \ref{17},  we have
\begin{align*}
\left\vert{_a^\mathbb{T}I_t^\alpha(_a^\mathbb{T}D_t^\alpha} u(t))\right\vert
=&\frac{1}{\Gamma(\alpha)}\bigg\vert\int_a^t(t-\sigma(s))^{\alpha-1}{_a^\mathbb{T}}D_t^\alpha u(s)\Delta s\bigg\vert\\
\leq&\frac{1}{\Gamma(\alpha)}\bigg(\int_a^t(t-\sigma(s))^{(\alpha-1)q}\Delta s\bigg)^{\frac{1}{q}}\|{_a^\mathbb{T}}D_t^\alpha u\|_{L_\Delta^p}\\
\leq&\frac{1}{\Gamma(\alpha)}\bigg(\int_a^t(t-s)^{(\alpha-1)q}ds\bigg)
^{\frac{1}{q}}\|{_a^\mathbb{T}}D_t^\alpha u\|_{L_\Delta^p}\\
\leq&\frac{b^{\frac{1}{q}+\alpha-1}}{\Gamma(\alpha)((\alpha-1)q+1)
^{\frac{1}{q}}}\left\|_a^\mathbb{T}D_t^\alpha u\right\|_{L_\Delta^p}\\
=&\frac{b^{\alpha-\frac{1}{p}}}{\Gamma(\alpha)((\alpha-1)q+1)
^{\frac{1}{q}}}\left\|_a^\mathbb{T}D_t^\alpha u\right\|_{L_\Delta^p}.
\end{align*}
The proof is complete.
\end{proof}

\begin{remark}
\begin{itemize}
  \item [$(i)$]
  According to $(\ref{33})$, we can consider $W_{\Delta,a^+}^{\alpha,p}$ with respect to the norm
  \begin{eqnarray}\label{37}
  \|u\|_{W_{\Delta,a^+}^{\alpha,p}}^p=\|_a^\mathbb{T}D_t^\alpha u\|_{L_\Delta^p}^p=\bigg(\int_{J^0}\left\vert{_a^\mathbb{T}D_t^\alpha} u(t)\right\vert^p\Delta t\bigg)^{\frac{1}{p}}
  \end{eqnarray}
  in the following analysis.
  \item [$(ii)$]
  It follows from $(\ref{33})$ and $(\ref{34})$ that $W_{\Delta,a^+}^{\alpha,p}$ is continuously immersed into $C(J, \mathbb{R}^N)$ with the natural norm $\|\cdot\|_\infty$.
\end{itemize}
\end{remark}

\begin{proposition}\label{39}
Let $0<\alpha\leq1$ and $1<p<\infty$. Assume that $\alpha>\frac{1}{p}$ and the sequence $\{u_k\}\subset W_{\Delta,a^+}^{\alpha,p}$ converges weakly to $u$ in $W_{\Delta,a^+}^{\alpha,p}$. Then, $u_k\rightarrow u$ in $C(J, \mathbb{R}^N)$, i.e., $\|u-u_k\|_\infty=0$, as $k\rightarrow\infty$.
\end{proposition}

\begin{proof}
If $\alpha>\frac{1}{p}$, then by $(\ref{34})$ and $(\ref{37})$, the injection of $W_{\Delta,a^+}^{\alpha,p}$ into $C(J, \mathbb{R}^N)$, with its natural norm $\|\cdot\|_\infty$, is continuous, i.e., $u_k\rightarrow u$ in $W_{\Delta,a^+}^{\alpha,p}$, then $u_k\rightarrow u$ in $C(J, \mathbb{R}^N)$.

Since $u_k\rightharpoonup u$ in $W_{\Delta,a^+}^{\alpha,p}$, it follows that $u_k\rightharpoonup u$ in $C(J, \mathbb{R}^N)$. In fact, for any $h\in\left(C(J, \mathbb{R}^N)\right)^*$, if $u_k\rightarrow u$ in $W_{\Delta,a^+}^{\alpha,p}$, then $u_k\rightarrow u$ in $C(J, \mathbb{R}^N)$, and thus $h(u_k)\rightarrow h(u)$. Therefore, $h\in\left(W_{\Delta,a^+}^{\alpha,p}\right)^*$, which means that $\left(C(J, \mathbb{R}^N)\right)^*\subset\left(W_{\Delta,a^+}^{\alpha,p}\right)^*$. Hence, if $u_k\rightharpoonup u$ in $W_{\Delta,a^+}^{\alpha,p}$, then for any $h\in\left(C(J, \mathbb{R}^N)\right)^*$, we have $h\in\left(W_{\Delta,a^+}^{\alpha,p}\right)^*$, and thus $h(u_k)\rightarrow h(u)$, i.e., $u_k\rightharpoonup u$ in $C(J, \mathbb{R}^N)$.

By the Banach$-$Steinhaus theorem, $\{u_k\}$ is bounded in $W_{\Delta,a^+}^{\alpha,p}$ and, hence, in $C(J, \mathbb{R}^N)$. Now, we  prove that the sequence $\{u_k\}$ is equi$-$continuous. Let $\frac{1}{p}+\frac{1}{q}=1$ and $t_1, t_2\in J$ with $t_1\leq t_2$,for all $f\in L_\Delta^p(J, \mathbb{R}^N)$, by using Proposition \ref{17}, Proposition \ref{2}, Theorem \ref{20}, and noting $\alpha>\frac{1}{p}$, we have
{\setlength\arraycolsep{2pt}
\begin{eqnarray}\label{38}
&&\left\vert{_a^\mathbb{T}I_{t_1}^\alpha} f(t_1)-{_a^\mathbb{T}}I_{t_2}^\alpha f(t_2)\right\vert\nonumber\\
&=&\frac{1}{\Gamma(\alpha)}\bigg\vert\int_a^{t_1}(t_1-\sigma(s))^{\alpha-1}f(s)\Delta s-\int_a^{t_2}(t_2-\sigma(s))^{\alpha-1}f(s)\Delta s\bigg\vert\nonumber\\
&\leq&\frac{1}{\Gamma(\alpha)}\bigg\vert\int_a^{t_1}(t_1-\sigma(s))^{\alpha-1}f(s)\Delta s-\int_a^{t_1}(t_2-\sigma(s))^{\alpha-1}f(s)\Delta s\bigg\vert\nonumber\\
&&+\frac{1}{\Gamma(\alpha)}\bigg\vert\int_{t_1}^{t_2}(t_2-\sigma(s))^{\alpha-1}f(s)\Delta s\bigg\vert\nonumber\\
&\leq&\frac{1}{\Gamma(\alpha)}\int_a^{t_1}\left((t_1-\sigma(s))^{\alpha-1}
-(t_2-\sigma(s))^{\alpha-1}\right)\vert f(s)\vert\Delta s\nonumber\\
&&+\frac{1}{\Gamma(\alpha)}\int_{t_1}^{t_2}(t_2-\sigma(s))^{\alpha-1}\vert f(s)\vert\Delta s\nonumber\\
&\leq&\frac{1}{\Gamma(\alpha)}\bigg(\int_a^{t_1}\left((t_1-\sigma(s))^{\alpha-1}
-(t_2-\sigma(s))^{\alpha-1}\right)^q\Delta s\bigg)^{\frac{1}{q}}\|f\|_{L_\Delta^p}\nonumber\\
&&+\frac{1}{\Gamma(\alpha)}\bigg(\int_{t_1}^{t_2}(t_2-\sigma(s))^{(\alpha-1)q}\Delta s\bigg)^{\frac{1}{q}}\|f\|_{L_\Delta^p}\nonumber\\
&\leq&\frac{1}{\Gamma(\alpha)}\bigg(\int_a^{t_1}\left((t_1-\sigma(s))^{(\alpha-1)q}
-(t_2-\sigma(s))^{(\alpha-1)q}\right)\Delta s\bigg)^{\frac{1}{q}}\|f\|_{L_\Delta^p}\nonumber\\
&&+\frac{1}{\Gamma(\alpha)}\bigg(\int_{t_1}^{t_2}(t_2-\sigma(s))^{(\alpha-1)q}\Delta s\bigg)^{\frac{1}{q}}\|f\|_{L_\Delta^p}\nonumber\\
&\leq&\frac{1}{\Gamma(\alpha)}\bigg(\int_a^{t_1}\left((t_1-s)^{(\alpha-1)q}
-(t_2-s)^{(\alpha-1)q}\right)ds\bigg)^{\frac{1}{q}}\|f\|_{L_\Delta^p}\nonumber\\
&&+\frac{1}{\Gamma(\alpha)}\bigg(\int_{t_1}^{t_2}(t_2-s)^{(\alpha-1)q}ds\bigg)
^{\frac{1}{q}}\|f\|_{L_\Delta^p}\nonumber\\
&=&\frac{\|f\|_{L_\Delta^p}}{\Gamma(\alpha)\left(1+(\alpha-1)q\right)^{\frac{1}{q}}}
\bigg(t_1^{(\alpha-1)q+1}-t_2^{(\alpha-1)q+1}+(t_2-t_1)^{(\alpha-1)q+1}\bigg)
^{\frac{1}{q}}\nonumber\\
&&+\frac{\|f\|_{L_\Delta^p}}{\Gamma(\alpha)\left(1+(\alpha-1)q\right)^{\frac{1}{q}}}
\bigg((t_2-t_1)^{(\alpha-1)q+1}\bigg)^{\frac{1}{q}}\nonumber\\
&\leq&\frac{2\|f\|_{L_\Delta^p}}{\Gamma(\alpha)\left(1+(\alpha-1)q\right)^{\frac{1}{q}}}
(t_2-t_1)^{\alpha-1+\frac{1}{q}}\nonumber\\
&=&\frac{2\|f\|_{L_\Delta^p}}{\Gamma(\alpha)\left(1+(\alpha-1)q\right)^{\frac{1}{q}}}
(t_2-t_1)^{\alpha-\frac{1}{p}}.
\end{eqnarray}}
Therefore, the sequence $\{u_k\}$ is equi$-$continuous since, for $t_1, t_2\in J$, $t_1\leq t_2$, by applying $(\ref{38})$ and $(\ref{37})$, we have
\begin{align*}
\left\vert u_k(t_1)-u_k(t_2)\right\vert
=&\left\vert{_a^\mathbb{T}I_{t_1}^\alpha(_a^\mathbb{T}D_{t_1}^\alpha} u_k(t_1))-{_a^\mathbb{T}}I_{t_2}^\alpha(_a^\mathbb{T}D_{t_2}^\alpha u_k(t_2))\right\vert\\
\leq&\frac{2(t_2-t_1)^{\alpha-\frac{1}{p}}}{\Gamma(\alpha)\left(1+(\alpha-1)q\right)^{\frac{1}{q}}}
\|_a^\mathbb{T}D_t^\alpha u_k\|_{L_\Delta^p}\\
=&\frac{2(t_2-t_1)^{\alpha-\frac{1}{p}}}{\Gamma(\alpha)\left(1+(\alpha-1)q\right)^{\frac{1}{q}}}
\|_a^\mathbb{T}D_t^\alpha u_k\|_{L_\Delta^p}\\
\leq&\frac{2(t_2-t_1)^{\alpha-\frac{1}{p}}}{\Gamma(\alpha)((\alpha-1)q+1)
^{\frac{1}{q}}}\left\|_a^\mathbb{T}D_t^\alpha u\right\|_{L_\Delta^p}\\
=&\frac{2(t_2-t_1)^{\alpha-\frac{1}{p}}}{\Gamma(\alpha)((\alpha-1)q+1)
^{\frac{1}{q}}}\left\|u_k\right\|_{W_{\Delta,a^+}^{\alpha,p}}\\
\leq&C(t_2-t_1)^{\alpha-\frac{1}{p}},
\end{align*}
where $\frac{1}{p}+\frac{1}{q}=1$ and $C\in\mathbb{R}^+$ is a constant. By the Ascoli$-$Arzela theorem on time scales (Lemma \ref{21}),  $\{u_k\}$ is relatively compact in $C(J, \mathbb{R}^N)$. By the uniqueness of the weak limit in $C(J, \mathbb{R}^N)$, every uniformly convergent subsequence of $\{u_k\}$ converges uniformly on $J$ to $u$. The proof is complete.
\end{proof}

\begin{remark}\label{OR1}
It follows from Proposition \ref{39} that $W_{\Delta,a^+}^{\alpha,p}$ is compactly immersed into $C(J, \mathbb{R}^N)$ with the natural norm $\|\cdot\|_\infty$.
\end{remark}

\begin{theorem}\label{3.9}
Let $1<p<\infty$, $\frac{1}{p}<\alpha\leq1$, $\frac{1}{p}+\frac{1}{q}=1$, $L:J\times\mathbb{R}^N\times\mathbb{R}^N\rightarrow\mathbb{R}$, $(t,x,y)\mapsto L(t,x,y)$ satisfies
\begin{itemize}
  \item [$(i)$]
  for each $(x,y)\in\mathbb{R}^N\times\mathbb{R}^N$, $L(t,x,y)$ is $\Delta-$measurable in $t$;
  \item [$(ii)$]
  for $\Delta-$almost every $t\in J$, $L(t,x,y)$ is continuously differentiable in $(x,y)$.
\end{itemize}
If there exist  $m_1\in C(\mathbb{R}^+,\mathbb{R}^+)$, $m_2\in L_\Delta^1(J,\mathbb{R}^+)$ and $m_3\in L_\Delta^q(J,\mathbb{R}^+)$, $1<q<\infty$, such that, for $\Delta-$a.e. $t\in J$ and every $(x,y)\in\mathbb{R}^N\times\mathbb{R}^N$, one has
\begin{align*}
\vert L(t,x,y)\vert&\leq m_1(\vert x\vert)(m_2(t)+\vert y\vert^p),\\
\vert D_xL(t,x,y)\vert&\leq m_1(\vert x\vert)(m_2(t)+\vert y\vert^p),\\
\vert D_yL(t,x,y)\vert&\leq m_1(\vert x\vert)(m_3(t)+\vert y\vert^{p-1}).
\end{align*}
Then the functional $\chi$ defined by
\begin{eqnarray*}
\chi(u)=\int_{J^0}L(t,u(t),\,_a^\mathbb{T}D_t^\alpha u(t))\Delta t
\end{eqnarray*}
is continuously differentiable on $W_{\Delta,a^+}^{\alpha,p}$, and $\forall \,u,v\in W_{\Delta,a^+}^{\alpha,p}$, one has
\begin{eqnarray}\label{32}
\langle \chi'(u),v\rangle=\int_{J^0}\bigg[\left(D_xL(t,u(t),\,_a^\mathbb{T}D_t^\alpha u(t),v(t)\right)+\left(D_yL(t,u(t),\,_a^\mathbb{T}D_t^\alpha u(t),\,_a^\mathbb{T}D_t^\alpha v(t)\right)\bigg]\Delta t.\quad
\end{eqnarray}
\end{theorem}

\begin{proof}
It suffices to prove that $\chi$ has, at every point $u$, a directional derivative $\chi'(u)\in(W_{\Delta,a^+}^{\alpha,p})^*$ given by $(\ref{32})$ and that the mapping
\begin{eqnarray*}
\chi':W_{\Delta,a^+}^{\alpha,p}\ni u\mapsto \chi'(u)\in(W_{\Delta,a^+}^{\alpha,p})^*
\end{eqnarray*}
is continuous. The rest proof is similar to the proof of Theorem 1.4 in \cite{12}. We will omit it here. The proof is complete.
\end{proof}

\section{An application}
\setcounter{equation}{0}
\indent

As an application of the concepts we introduced and the results obtained in Section 3, in this section we will use critical point theory to study the solvability of a class of boundary value problems on time scales.
 More precisely, our goal is to study the following Kirchhoff-type fractional $p$-Laplacian systems on time scales with boundary condition(KFBVP$_{\mathbb{T}}$ for short):
\begin{equation}\label{O1}
\begin{cases}
\left(\beta+\varrho\int_{{[a,b)}_{\mathbb{T}}}\vert{^{\mathbb{T}}_a}D^\alpha_tu(t)\vert^p\Delta t\right)^{p-1}\,^{\mathbb{T}}_tD^\alpha_b\phi_p(^{\mathbb{T}}_aD^\alpha_tu(t))=\lambda (t)\nabla G(t,u(t)),\quad \Delta-a.e.\,\,t\in [a,b]_{\mathbb{T}},\\
u(a)=u(b)=0,
\end{cases}
\end{equation}
where $\beta,\varrho>0$ and $p>1$ are constants, $\lambda\in L^\infty_\Delta([a,b]_{\mathbb{T}},\mathbb{R}^+)$ with
$ess\sup\limits_{t\in[a,b]_{\mathbb{T}}}\lambda(t):=\lambda^0
>\lambda_0:=ess\inf\limits_{t\in[a,b]_{\mathbb{T}}}\lambda(t)>0$, $^{\mathbb{T}}_tD^\alpha_b$ and $^{\mathbb{T}}_aD^\alpha_t$ are the right and the left Riemann$-$Liouville fractional derivative operators of order $\alpha$ defined on $\mathbb{T}$ respectively, and $\phi_p:\mathbb{R}\rightarrow\mathbb{R}$ is the $p-$Laplacian(\cite{O1}) defined by
\begin{equation*}
\phi_p(y)=
\left\{
\begin{aligned}
&\vert y\vert^{p-2}y,&\quad if\,\,y\neq0,\\
&0,&\quad if\,\,y=0.
\end{aligned}
\right.
\end{equation*}
And $\nabla G\in C([a,b]_{\mathbb{T}}\times\mathbb{R}, \mathbb{R})$ denotes the gradient of $G(t,x)$ in $x$.
When $\mathbb{T}=\mathbb{R}$, FBVP$_{\mathbb{T}}$ (\ref{O1}) reduces to the following Kirchhoff-type fractional $p$-Laplacian systems
\begin{equation*}
\begin{cases}
\left(\beta+\varrho\int_a^b\vert{_aD^\alpha_t}u(t)\vert^pd t\right)^{p-1}\,_tD^\alpha_b\phi_p(_aD^\alpha_tu(t))=\lambda (t)\nabla G(t,u(t)),\quad a.e.\,\,t\in [a,b],\\
u(a)=u(b)=0,
\end{cases}
\end{equation*}
When $\mathbb{T}=\mathbb{R}$ and $\lambda(t)=\lambda\in(0,+\infty)$, FBVP$_{\mathbb{T}}$ (\ref{O1}) reduces to the following Kirchhoff-type fractional $p$-Laplacian systems
\begin{equation*}
\begin{cases}
\left(\beta+\varrho\int_a^b\vert{_aD^\alpha_t}u(t)\vert^pd t\right)^{p-1}\,_tD^\alpha_b\phi_p(_aD^\alpha_tu(t))=\lambda \nabla G(t,u(t)),\quad a.e.\,\,t\in [a,b],\\
u(a)=u(b)=0,
\end{cases}
\end{equation*}
When $\mathbb{T}=\mathbb{R}$, $\lambda(t)=1$, our results further reduce to the following problems
\begin{equation*}
\begin{cases}
\left(\beta+\varrho\int_a^b\vert{_aD^\alpha_t}u(t)\vert^pd t\right)^{p-1}\,_tD^\alpha_b\phi_p(_aD^\alpha_tu(t))=\nabla G(t,u(t)),\quad a.e.\,\,t\in [a,b],\\
u(a)=u(b)=0,
\end{cases}
\end{equation*}
which has been studied by \cite{O2}. So, in short, our results improved and generalized \cite{O2}.

\begin{definition}\label{OD1} \cite{12}
Let $E$ be a real Banach space and $\varphi\in C^1(E,\mathbb{R})$. If any sequence $u_k\}\subset E$ for which $\varphi(u_k)$ is bounded and $\varphi'(u_k)\rightarrow0$ as $k\rightarrow\infty$ possesses a convergent subsequence in $E$, then we say that $\varphi$ satisfies the $(PS)$ condition.
\end{definition}

\begin{lemma}\label{OL1} \cite{O3}
Let $E$ be a real Banach space and $\varphi\in C^1(E,\mathbb{R})$ satisfying the $(PS)$ condition. Assume that $\varphi(0)=0$ and the following conditions:
\begin{itemize}
  \item [$(A_1)$]
  there are constants $\rho,\sigma>0$ such that $\varphi|_{\partial B_\rho(0)}\geq\sigma$;
  \item [$(A_2)$]
  there exists an $e\in E\setminus\overline{B_\rho(0)}$ such that $\varphi(e)\leq0$.
\end{itemize}
Then, $\varphi$ possesses a critical value $c\geq\sigma$. Furthermore, $c$ can be characterized as
\begin{eqnarray*}
c=\inf_{\nu\in\Gamma}\max_{s\in[0,1]}\varphi(\nu(s)),
\end{eqnarray*}
where
\begin{eqnarray*}
\Gamma=\{\nu\in C([0,1],E)\vert\nu(0)=0,\nu(1)=e\}.
\end{eqnarray*}
\end{lemma}

\begin{lemma}\label{OL2} \cite{12}
Let $E$ be a real Banach space and $\varphi\in C^1(E,\mathbb{R})$ satisfying the $(PS)$ condition. If $\varphi$ is bounded from below, then $c=\inf\limits_{E}\varphi$ is a critical value of $\varphi$.
\end{lemma}

For the sake of the infinitely many critical points of $\varphi$, one introduces the genus properties as follows. First, we let
\begin{eqnarray*}
&\Xi&=\{A\subset E-\{0\}\vert A\,\,\mathrm{is}\,\,\mathrm{closed}\,\,\mathrm{in}\,\,E\,\,\mathrm{and}\,\,
\mathrm{symmetric}\,\,\mathrm{with}\,\,\mathrm{respect}\,\,\mathrm{to}\,\,0\},\\
&K_c&=\{u\in E\vert\varphi(u)=c,\varphi'(u)=0\},\\
&\varphi^c&=\{u\in E\vert\varphi(u)\leq c\}.
\end{eqnarray*}

\begin{definition}\label{def1} \cite{O3}
For $A\in\Xi$, we say that the genus of $A$ is $n$ denoted by $\gamma(A)=n$ if there is an odd map $P\in C(A,\mathbb{R}^N\setminus\{0\})$ and $n$ is the smallest integer with this property.
\end{definition}

\begin{lemma}\label{OL3} \cite{O3}
Let $\varphi$ be an even $C^1$ functional on $E$ and satisfy the $(PS)$ condition. For any $n\in\mathbb{N}$, set
\begin{eqnarray*}
&\Xi_n&=\{A\in\Xi\vert\gamma(A)\geq n\},\\
&c_n&=\inf_{A\in\Xi_n}\sup_{u\in A}\varphi(u).
\end{eqnarray*}
\begin{itemize}
  \item [$(i)$]
  If $\Xi_n\neq0$ and $c_n\in\mathbb{R}$, then $c_n$ is a critical value of $\varphi$;
  \item [$(ii)$]
  If there exists $l\in \mathbb{N}$ such that $c_n=c_{n+1}=\cdots=c_{n+l}=c\in\mathbb{R}$ and $c\neq\varphi(0)$, then $\gamma(K_c)\geq l+1$.
\end{itemize}
\end{lemma}

\begin{remark}\label{le1} \cite{O3}
In view of Remark 7.3 in \cite{O3}, one sees that if $K_c\in\Xi$ and $\gamma(K_c)$ contains infinitely many distinct points. In other words, $\varphi$ has infinitely many distinct critical points in $E$.
\end{remark}

There have been many results using critical point theory  to study boundary value problems of fractional differential equations (\cite{f1,f2,f3,f4,f5,f6,31}) and dynamic equations on time scales (\cite{d1,d2,d3,d4,d5}), but the results of using critical point theory to study boundary value problems of fractional dynamic equations on time scales are still rare \cite{2a}.
This section will explain that critical point theory is an effective way to deal with the existence of  solutions of \eqref{O1} on  time scales.

In this section, we let $N=1$. For purpose of the presence and proof of our main results, let's first define the functional $\varphi:W_{\Delta,a^+}^{\alpha,p}\rightarrow\mathbb{R}$ by
\begin{align}\label{Oe1}
\varphi(u)=\frac{1}{\varrho p^2}\left(\beta+\varrho\int_{J^0}\vert{^{\mathbb{T}}_aD_t^\alpha} u(t)\vert^p\Delta t\right)^p
-\int_{J^0}\lambda (t) G(t,u(t))\Delta t-\frac{\beta^p}{\varrho p^2}
\end{align}
It is easy to testify from (\ref{33}), (\ref{37}) and $g\in C(J\times\mathbb{R},\mathbb{R})$ that the functional $\varphi$ is well defined on $W_{\Delta,a^+}^{\alpha,p}$ and $\varphi\in C(W_{\Delta,a^+}^{\alpha,p},\mathbb{R})$. Moreover, for $\forall u,v\in W_{\Delta,a^+}^{\alpha,p}$, one obtains
\begin{align}\label{Oe2}
\langle\varphi'(u),v\rangle=(\beta+\varrho\|u\|^p)^{p-1}\int_{J^0}\phi_p(^{\mathbb{T}}_aD^\alpha_t
u(t))^{\mathbb{T}}_aD^\alpha_tv(t)\Delta t-\int_{J^0} \lambda (t)\nabla G(t,u(t))v(t)\Delta t,
\end{align}
which yields
\begin{align}\label{Oe3}
\langle\varphi'(u),u\rangle=(\beta+\varrho\|u\|^p)^{p-1}\|u\|^p-\int_{J^0} \lambda (t)\nabla G(t,u(t))u(t)\Delta t.
\end{align}

Now, it is time for us to present and prove our main results as follows:

\begin{theorem}\label{OT1}
Let $\alpha\in\left(\frac{1}{p},1\right]$, suppose that $G$ satisfies the following conditions:
\begin{itemize}
  \item [$(G_1)$]
  $G(t,x)$ is $\Delta-$ measurable and continuously differentiable in $x$ for $t\in J$ and there exist $a\in C(\mathbb{R}^+,\mathbb{R}^+)$, $b\in L_\Delta^1(J,\mathbb{R}^+)$ such that
  \begin{eqnarray}\label{37}
  \vert G(t,x)\vert\leq a(\vert x\vert)b(t),\quad \vert\nabla G(t,x)\vert\leq a(\vert x\vert)b(t)
  \end{eqnarray}
  for all $x\in\mathbb{R}$ and $t\in J$.
  \item [$(G_2)$]
  There are two constants $\mu>p^2$, $M>0$ such that
  \begin{eqnarray*}
  0<\mu \,G(t,x)\leq x\,\nabla G(t,x),\quad \forall t\in J\,\,\mathrm{and}\,\,\vert x\vert\geq M.
  \end{eqnarray*}
  \item [$(G_3)$]
  $\nabla G(t,x)=o(\vert x\vert^{p-1})$ as $\vert x\vert\rightarrow0$ uniformly for $t\in J$.
\end{itemize}
Then, KFBVP$_{\mathbb{T}}$ (\ref{O1}) has at least one nontrivial weak solution.
\end{theorem}

\begin{proof}
It is clear that $\varphi(0)=0$, $\varphi\in C^1(W_{\Delta,a^+}^{\alpha,p},\mathbb{R})$, where $W_{\Delta,a^+}^{\alpha,p}$ is a real Banach space from Theorem \ref{3.6}, So now, we are in a position by using Mountain pass theorem (Lemma \ref{OL1}) to prove that
\begin{itemize}
  \item [$\mathrm{step}\,\,1.$]
  $\varphi$ satisfies the $(PS)$ condition in $W_{\Delta,a^+}^{\alpha,p}$. The argument is as follows:
  Let $\{u_k\}\subset W_{\Delta,a^+}^{\alpha,p}$ be a sequence such that
  \begin{eqnarray}\label{Oe4}
  &\vert\varphi(u_k)\vert\leq K,&\nonumber\\
  &\varphi'(u_k)\rightarrow0\quad \mathrm{as}\,\, k\rightarrow\infty,&
  \end{eqnarray}
  where $K>0$ is a constant. We first prove that $\{u_k\}$ is bounded in $W_{\Delta,a^+}^{\alpha,p}$. From the continuity of $\mu G(t,x)-xg(t,x)$, we obtain that there is a constant $c>0$ such that
  \begin{eqnarray*}
  G(t,x)\leq \frac{1}{\mu}x\nabla G(t,x)+c,\quad \forall t\in J,\,\, \vert x\vert\leq M
  \end{eqnarray*}
  Combining with $(\mathbf{G_2})$, we obtain that
  \begin{eqnarray}\label{Oe5}
  G(t,x)\leq \frac{1}{\mu}x\nabla G(t,x)+c,\quad \forall (t,x)\in J\times\mathbb{R}.
  \end{eqnarray}
  Hence, taking account of (\ref{37}), (\ref{Oe1}), (\ref{Oe3}), (\ref{Oe4}), (\ref{Oe5}) and Lemma \ref{OL4}, we have
  {\setlength\arraycolsep{2pt}
  \begin{eqnarray}\label{Oe6}
  &&K\nonumber\\
  &\geq&\varphi(u_k)\nonumber\\
  &=&\frac{1}{\varrho p^2}\left(\beta+\varrho\int_{J^0}\vert{^{\mathbb{T}}_a}D_t^\alpha u_k(t)\vert^p\Delta t\right)^p-\int_{J^0} \lambda (t)G(t,u_k(t))\Delta t-\frac{\beta^p}{\varrho p^2}\nonumber\\
  &=&\frac{1}{\varrho p^2}(\beta+\varrho\|u_k\|^p)^p-\int_{J^0} \lambda (t)G(t,u_k(t))\Delta t-\frac{\beta^p}{\varrho p^2}\nonumber\\
  &\geq&\frac{1}{\varrho p^2}(\beta+\varrho\|u_k\|^p)^p-\int_{J^0}\left[\frac{1}{\mu}u_k(t) \nabla G(t,u_k(t))+c\right]\Delta t-\frac{\beta^p}{\varrho p^2}\nonumber\\
  &=&\frac{1}{\varrho p^2}(\beta+\varrho\|u_k\|^p)^p+\frac{1}{\mu}\langle\varphi'(u_k),u_k\rangle
  -\frac{1}{\mu}(\beta+\varrho\|u_k\|^p)^{p-1}-c(b-a)-\frac{\beta^p}{\varrho p^2}\nonumber\\
  &\geq&(\beta+\varrho\|u_k\|^p)^{p-1}\left[\left(\frac{1}{p^2}-\frac{1}{\mu}\right)\|u_k\|^p
  +\frac{\beta}{\varrho p^2}\right]-\frac{1}{\mu}\|\varphi'(u_k)\|_{(W_{\Delta,a^+}^{\alpha,p})^*}\|u_k\|\nonumber\\
  &&-cb-\frac{\beta^p}{\varrho p^2},
  \end{eqnarray}}
  which together with $\varphi'(u_k)\rightarrow0$ as $k\rightarrow\infty$ yields
  \begin{eqnarray}\label{Oe7}
  K\geq (\beta+\varrho\|u_k\|^p)^{p-1}\left[\left(\frac{1}{p^2}-\frac{1}{\mu}\right)\|u_k\|^p
  +\frac{\beta}{\varrho p^2}\right]-\|u_k\|-cb-\frac{\beta^p}{\varrho p^2}.
  \end{eqnarray}
  Then, combining with $\mu>p^2$ and proof by contradiction, we know that $\{u_k\}$ is bounded in $W_{\Delta,a^+}^{\alpha,p}$.

  Because $W_{\Delta,a^+}^{\alpha,p}$ is a reflexive Banach space (Theorem \ref{31} and Theorem \ref{3.6}), going if necessary to a subsequence, we can assume $u_k\rightharpoonup u$ in $W_{\Delta,a^+}^{\alpha,p}$. As a result, in view of $\varphi'(u_k)\rightarrow0$ as $k\rightarrow\infty$ and the definition of weak convergence, one sees
  \begin{align}\label{Oe8}
  \langle\varphi'(u_k)-\varphi'(u),u_k-u\rangle
=&\langle\varphi'(u_k),u_k-u\rangle-\langle\varphi'(u),u_k-u\rangle\nonumber\\
\leq&\|\varphi'(u_k)\|_{(W_{\Delta,a^+}^{\alpha,p})^*}\|u_k-u\|-\langle\varphi'(u),u_k-u\rangle\nonumber\\
\leq&\|\varphi'(u_k)\|_{(W_{\Delta,a^+}^{\alpha,p})^*}(\|u_k\|+\|u\|)-\langle\varphi'(u),u_k-u\rangle\nonumber\\
\rightarrow&0,\quad \mathrm{as}\,\,k\rightarrow\infty.
  \end{align}
  Furthermore, it follows from (\ref{34}), (\ref{37}) and Remark \ref{OR1} that $\{u_k\}$ is bounded in $C(J,\mathbb{R})$ and $\|u_k-u\|_\infty\rightarrow0$, as $k\rightarrow\infty$. Therefore, there is a constant $c_1>0$ such that
  \begin{eqnarray}\label{Oe9}
  \vert\nabla G(t,u_k(t))-\nabla G(t,u(t))\vert\leq c_1,\quad\forall t\in J,
  \end{eqnarray}
  which yields
  {\setlength\arraycolsep{2pt}
  \begin{eqnarray}\label{Oe10}
  &&\left\vert\int_{J^0}(\nabla G(t,u_k(t))-\nabla G(t,u(t)))(u_k(t)-u(t))\Delta t\right\vert\nonumber\\
  &\leq&c_1b\|u_k-u\|_\infty\nonumber\\
  &\rightarrow&0,\quad \mathrm{as}\,\, k\rightarrow\infty.
  \end{eqnarray}}
  Furthermore, it follows from the boundedness of $\{u_k\}$ in $W_{\Delta,a^+}^{\alpha,p}$ that
  {\setlength\arraycolsep{2pt}
  \begin{eqnarray}\label{Oe11}
  &&\left[(\beta+\varrho\|u_k\|^p)^{p-1}-(\beta+\varrho\|u\|^p)^{p-1}\right]
  \int_{J^0}\phi_p(^{\mathbb{T}}_aD^\alpha_t
  u(t))(\,^{\mathbb{T}}_aD^\alpha_tu_k(t)-\,^{\mathbb{T}}_aD^\alpha_tu(t))\Delta t\nonumber\\
  &=&\left[(\beta+\varrho\|u_k\|^p)^{p-1}-(\beta+\varrho\|u\|^p)^{p-1}\right]
  \left\langle\frac{1}{p}\int_{J^0}|^{\mathbb{T}}_aD_t^\alpha u(t)|^p\Delta t,u_k-u\right\rangle\nonumber\\
  &\rightarrow&0,\quad \mathrm{as}\,\, k\rightarrow\infty.
  \end{eqnarray}}
  In consideration of (\ref{Oe2}), one obtains
  {\setlength\arraycolsep{2pt}
  \begin{eqnarray}\label{Oe12}
  &&\langle\varphi'(u_k)-\varphi'(u),u_k-u\rangle+\int_{J^0}\lambda (t)(\nabla G(t,u_k(t))-\nabla G(t,u(t)))(u_k(t)-u(t))\Delta t\nonumber\\
  &=&(\beta+\varrho\|u_k\|^p)^{p-1}
  \int_{J^0}\phi_p(^{\mathbb{T}}_aD^\alpha_t
  u_k(t))(\,^{\mathbb{T}}_aD^\alpha_tu_k(t)-\,^{\mathbb{T}}_aD^\alpha_tu(t))\Delta t\nonumber\\
  &&-(\beta+\varrho\|u\|^p)^{p-1}\int_{J^0}\phi_p(^{\mathbb{T}}_aD^\alpha_t
  u(t))(\,^{\mathbb{T}}_aD^\alpha_tu_k(t)-\,^{\mathbb{T}}_aD^\alpha_tu(t))\Delta t\nonumber\\
  &=&(\beta+\varrho\|u_k\|^p)^{p-1}
  \int_{J^0}(\phi_p(^{\mathbb{T}}_aD^\alpha_t
  u_k(t))-\phi_p(^{\mathbb{T}}_aD^\alpha_t
  u(t)))(\,^{\mathbb{T}}_aD^\alpha_tu_k(t)-\,^{\mathbb{T}}_aD^\alpha_tu(t))\Delta t\nonumber\\
  &&+\left[(\beta+\varrho\|u_k\|^p)^{p-1}-(\beta+\varrho\|u\|^p)^{p-1}\right]\nonumber\\
  &&\times\int_{J^0}(\phi_p(^{\mathbb{T}}_aD^\alpha_t
  u(t))(\,^{\mathbb{T}}_aD^\alpha_tu_k(t)-(\,^{\mathbb{T}}_aD^\alpha_tu(t))\Delta t,
  \end{eqnarray}}
  which together with (\ref{Oe8})-(\ref{Oe12}) yields
  \begin{eqnarray}\label{Oe13}
  \int_{J^0}(\phi_p(^{\mathbb{T}}_aD^\alpha_t
  u_k(t))-\phi_p(^{\mathbb{T}}_aD^\alpha_t
  u(t)))(\,^{\mathbb{T}}_aD^\alpha_tu_k(t)-\,^{\mathbb{T}}_aD^\alpha_tu(t))\Delta t \rightarrow0,\quad \mathrm{as}\,\, k\rightarrow\infty.
  \end{eqnarray}
  Taking into consideration of (2.10) in \cite{O4}, we can find two positive constants $c_2$, $c_3$ such that
    {\setlength\arraycolsep{2pt}
  \begin{eqnarray}\label{Oe14}
  &&\int_{J^0}(\phi_p(^{\mathbb{T}}_aD^\alpha_t
  u_k(t))-\phi_p(^{\mathbb{T}}_aD^\alpha_t
  u(t)))(\,^{\mathbb{T}}_aD^\alpha_tu_k(t)-\,^{\mathbb{T}}_aD^\alpha_tu(t))\Delta t\nonumber\\
  &\geq&
\begin{cases}
c_2\int_{J^0}\vert{^{\mathbb{T}}_aD^\alpha_t}u_k(t)-{^{\mathbb{T}}_aD^\alpha_t}u(t)\vert^p\Delta t,\quad p\geq2,\\
c_3\int_{J^0}\frac{\vert{^{\mathbb{T}}_aD^\alpha_t}u_k(t)-{^{\mathbb{T}}_aD^\alpha_t}u(t)\vert^2}
{(\vert{^{\mathbb{T}}_aD^\alpha_t}u_k(t)\vert+\vert{^{\mathbb{T}}_aD^\alpha_t}u(t)\vert)^{p-2}}\Delta t,\quad 1<p<2.
\end{cases}
  \end{eqnarray}}
When $1<p<2$, with an eye to Proposition \ref{17} and $(\vert x\vert+\vert y\vert)^p\leq2^{p-1}(\vert x\vert^p+\vert y\vert^p)(\,\,\forall x,y\in\mathbb{R},\,\,p>0)$, one obtains
{\setlength\arraycolsep{2pt}
  \begin{eqnarray}\label{Oe15}
  &&\int_{J^0}\vert{^{\mathbb{T}}_aD^\alpha_t}u_k(t)-{^{\mathbb{T}}_aD^\alpha_t}u(t)\vert^p\Delta t\nonumber\\
  &=&
  \int_{J^0}\frac{\vert{^{\mathbb{T}}_aD^\alpha_t}u_k(t)-{^{\mathbb{T}}_aD^\alpha_t}u(t)\vert^p}
  {(\vert{^{\mathbb{T}}_aD^\alpha_t}u_k(t)\vert+\vert{^{\mathbb{T}}_aD^\alpha_t}u(t)\vert)^{\frac{p(2-p)}{2}}}
  (\vert{^{\mathbb{T}}_aD^\alpha_t}u_k(t)\vert+\vert{^{\mathbb{T}}_aD^\alpha_t}u(t)\vert)^{\frac{p(2-p)}{2}}\Delta t\nonumber\\
  &\leq&
  \left\{\int_{J^0}\left[\frac{\vert{^{\mathbb{T}}_aD^\alpha_t}u_k(t)-{^{\mathbb{T}}_aD^\alpha_t}u(t)\vert^p}
  {(\vert{^{\mathbb{T}}_aD^\alpha_t}u_k(t)\vert+\vert{^{\mathbb{T}}_aD^\alpha_t}u(t)\vert)^{\frac{p(2-p)}{2}}}\right]
  ^{\frac{2}{p}}\Delta t\right\}^{\frac{p}{2}}\nonumber\\
  &&\times\left\{\int_{J^0}\left[(\vert{^{\mathbb{T}}_aD^\alpha_t}u_k(t)\vert+\vert{^{\mathbb{T}}_aD^\alpha_t}u(t)\vert)
  ^{\frac{p(2-p)}{2}}\right]^{\frac{2}{2-p}}\Delta t\right\}^{\frac{2-p}{2}}\nonumber\\
  &=&\left[\int_{J^0}\frac{\vert{^{\mathbb{T}}_aD^\alpha_t}u_k(t)-{^{\mathbb{T}}_aD^\alpha_t}u(t)\vert^2}
  {(\vert{^{\mathbb{T}}_aD^\alpha_t}u_k(t)\vert+\vert{^{\mathbb{T}}_aD^\alpha_t}u(t)\vert)^{2-p}}\Delta t\right]^{\frac{p}{2}}\nonumber\\
  &&\times\left[\int_{J^0}(\vert{^{\mathbb{T}}_aD^\alpha_t}u_k(t)\vert+\vert{^{\mathbb{T}}_aD^\alpha_t}u(t)\vert)
  ^p\Delta t\right]^{\frac{2-p}{2}}\nonumber\\
  &\leq&\left[\int_{J^0}\frac{\vert{^{\mathbb{T}}_aD^\alpha_t}u_k(t)-{^{\mathbb{T}}_aD^\alpha_t}u(t)\vert^2}
  {(\vert{^{\mathbb{T}}_aD^\alpha_t}u_k(t)\vert+\vert{^{\mathbb{T}}_aD^\alpha_t}u(t)\vert)^{2-p}}\Delta t\right]^{\frac{p}{2}}\nonumber\\
  &&\times\left[\int_{J^0}2^{p-1}(\vert{^{\mathbb{T}}_aD^\alpha_t}u_k(t)\vert^p+\vert{^{\mathbb{T}}_aD^\alpha_t}u(t)\vert^p)
  \Delta t\right]^{\frac{2-p}{2}}\nonumber\\
  &=&\left[\int_{J^0}\frac{\vert{^{\mathbb{T}}_aD^\alpha_t}u_k(t)-{^{\mathbb{T}}_aD^\alpha_t}u(t)\vert^2}
  {(\vert{^{\mathbb{T}}_aD^\alpha_t}u_k(t)\vert+\vert{^{\mathbb{T}}_aD^\alpha_t}u(t)\vert)^{2-p}}\Delta t\right]^{\frac{p}{2}}2^{\frac{(p-1)(2-p)}{2}}(\|u_k\|^p+\|u\|^p)^{\frac{2-p}{2}}.
  \end{eqnarray}}
Therefore, we have
  {\setlength\arraycolsep{2pt}
  \begin{eqnarray}\label{Oe16}
  &&\int_{J^0}\frac{\vert{^{\mathbb{T}}_aD^\alpha_t}u_k(t)-{^{\mathbb{T}}_aD^\alpha_t}u(t)\vert^2}
  {(\vert{^{\mathbb{T}}_aD^\alpha_t}u_k(t)\vert+\vert{^{\mathbb{T}}_aD^\alpha_t}u(t)\vert)^{2-p}}\Delta t\nonumber\\
  &\geq&\left[\frac{1}{2^{\frac{(p-1)(2-p)}{2}}(\|u_k\|^p+\|u\|^p)^{\frac{2-p}{2}}}\|u_k-u\|^p\right]
  ^{\frac{2}{p}}\nonumber\\
  &=&2^{\frac{(p-1)(p-2)}{p}}(\|u_k\|^p+\|u\|^p)^{\frac{p-2}{p}}\|u_k-u\|^2,
  \end{eqnarray}}
which together with (\ref{Oe14}) implies
  {\setlength\arraycolsep{2pt}
  \begin{eqnarray}\label{Oe17}
  &&\int_{J^0}(\phi_p(^{\mathbb{T}}_aD^\alpha_t
  u_k(t))-\phi_p(^{\mathbb{T}}_aD^\alpha_t
  u(t)))(\,^{\mathbb{T}}_aD^\alpha_tu_k(t)-\,^{\mathbb{T}}_aD^\alpha_tu(t))\Delta t\nonumber\\
  &\geq&c_3\int_{J^0}\frac{\vert{^{\mathbb{T}}_aD^\alpha_t}u_k(t)-{^{\mathbb{T}}_aD^\alpha_t}u(t)\vert^2}
  {(\vert{^{\mathbb{T}}_aD^\alpha_t}u_k(t)\vert+\vert{^{\mathbb{T}}_aD^\alpha_t}u(t)\vert)^{p-2}}\Delta t\nonumber\\
  &\geq&c_32^{\frac{(p-1)(p-2)}{p}}(\|u_k\|^p+\|u\|^p)^{\frac{p-2}{p}}\|u_k-u\|^2,\quad1<p<2.
  \end{eqnarray}}
When $p>2$, taking (\ref{Oe14}) into account, one obtains
  {\setlength\arraycolsep{2pt}
  \begin{eqnarray}\label{Oe18}
  &&\int_{J^0}(\phi_p(^{\mathbb{T}}_aD^\alpha_t
  u_k(t))-\phi_p(^{\mathbb{T}}_aD^\alpha_t
  u(t)))(\,^{\mathbb{T}}_aD^\alpha_tu_k(t)-\,^{\mathbb{T}}_aD^\alpha_tu(t))\Delta t\nonumber\\
  &\geq&c_2\|u_k-u\|^p,\quad p>2.
  \end{eqnarray}}
As a consequence, combining with (\ref{Oe13}), (\ref{Oe17}) and (\ref{Oe18}), one sees
  \begin{eqnarray}\label{Oe19}
  \|u_k-u\|\rightarrow0,\quad \mathrm{as}\,\, k\rightarrow\infty.
  \end{eqnarray}
Therefore, $\varphi$ satisfies the $(PS)$ condition in $W_{\Delta,a^+}^{\alpha,p}$.

  \item [$\mathrm{step}\,\,2.$]
  $\varphi$ satisfies the $(\mathbf{A_1})$ condition in Lemma \ref{OL1}, which can be explained by the following reasons:

Taking $(\mathbf{G_3})$ into account, we can find two positive constants $\varepsilon'\in(0,1)$ and $\delta$ such that
  \begin{eqnarray}\label{Oe20}
  G(t,x)\leq\frac{(1-\varepsilon')\beta^{p-1}}{\lambda^0p\frac{b^{\alpha p}}{\Gamma^p(\alpha+1)}}\vert x\vert^p,\quad\forall\,t\in J,\,\,\mathrm{with}\,\,\vert x\vert\leq\delta.
  \end{eqnarray}
Let $\rho=\frac{\delta}{\frac{b^{\alpha-\frac{1}{p}}}{\Gamma(\alpha)((\alpha-1)q+1)^{\frac{1}{q}}}}$ and $\delta=\frac{\varepsilon'\beta^{p-1}\rho^p}{p}$. Hence, taking (\ref{37}) into consideration, one has
  \begin{eqnarray}\label{Oe21}
  \|u\|_\infty\leq\frac{b^{\alpha-\frac{1}{p}}}{\Gamma(\alpha)((\alpha-1)q+1)^{\frac{1}{q}}}\|u\|,
  \quad\forall\,u\in W_{\Delta,a^+}^{\alpha,p},\,\, \mathrm{with}\,\,\|u\|=\rho,
  \end{eqnarray}
which together with (\ref{33}), (\ref{37}), (\ref{Oe1}) and (\ref{Oe20}) implies
\begin{align}\label{Oe22}
\varphi(u)
=&\frac{1}{\varrho p^2}\left(\beta+\varrho\int_{J^0}\vert{^{\mathbb{T}}_aD_t^\alpha } u(t)\vert^p\Delta t\right)^p
-\int_{J^0} \lambda (t)G(t,u(t))\Delta t-\frac{\beta^p}{\varrho p^2}\nonumber\\
=&\frac{1}{\varrho p^2}(\beta+\varrho\|u\|^p)^p
-\int_{J^0} \lambda (t)G(t,u(t))\Delta t-\frac{\beta^p}{\varrho p^2}\nonumber\\
\geq&\frac{\beta^{p-1}}{p}\|u\|^p-\lambda^0\frac{(1-\varepsilon')\beta^{p-1}}{\lambda^0p\frac{b^{\alpha p}}{\Gamma^p(\alpha+1)}}\int_{J^0} \vert u(t)\vert^p\Delta t\nonumber\\
\geq&\frac{\beta^{p-1}}{p}\|u\|^p-\lambda^0\frac{(1-\varepsilon')\beta^{p-1}}{\lambda^0p\frac{b^{\alpha p}}{\Gamma^p(\alpha+1)}}\frac{b^{\alpha p}}{\Gamma^p(\alpha+1)}\|_a^\mathbb{T}D_t^\alpha u\|_{L_\Delta^p}^p\nonumber\\
=&\frac{\beta^{p-1}}{p}\|u\|^p-\lambda^0\frac{(1-\varepsilon')\beta^{p-1}}{\lambda^0p}\|u\|^p\nonumber\\
=&\frac{\varepsilon'\beta^{p-1}}{p}\|u\|^p\nonumber\\
=&\sigma,\quad\forall\,u\in W_{\Delta,a^+}^{\alpha,p},\,\, \mathrm{with}\,\,\|u\|=\rho,
\end{align}
which means that the $(\mathbf{A_1})$ condition in Lemma \ref{OL1} is satisfied.

  \item [$\mathrm{step}\,\,3.$]
  $\varphi$ satisfies the $(\mathbf{A_2})$ condition in Lemma \ref{OL1}. Here are some reasons why:

For $s\in\mathbb{R}$, $\vert x\vert\geq M$ and $t\in J$, let
  \begin{eqnarray}\label{Oe23}
  F(s)=G(t,sx),\quad\quad H(s)=F'(s)-\frac{\mu}{s}F(s)
  \end{eqnarray}
In view of $(\mathbf{G_2})$, when $s\geq\frac{M}{\vert x\vert}$, one obtains
  \begin{eqnarray*}
  H(s)=\frac{\nabla G(t,sx)sx-\mu G(t,sx)}{s}\geq0
  \end{eqnarray*}
In addition, taking the expression of $F(\cdot)$ and $H(\cdot)$ in (\ref{Oe23}) into account, we can easily obtain that $F(s)$ satisfies
  \begin{eqnarray*}
  F'(s)=H(s)+\frac{\mu}{s}F(s)
  \end{eqnarray*}
Therefore, when $s\geq\frac{M}{\vert x\vert}$, we have
  \begin{eqnarray*}
  G(t,sx)=s^\mu\left[G(t,x)+\int_1^s\tau^{-\mu}H(\tau)d\tau\right].
  \end{eqnarray*}
So, for $\vert x\vert\geq M$ and $t\in J$, together with $(\mathbf{G_1})$, one obtains
  \begin{eqnarray*}
  \left(\frac{M}{\vert x\vert}\right)^\mu G(t,x)\leq G\left(t,x\frac{M}{\vert x\vert}\right)\leq\max_{\vert x\vert\leq M}a(\vert x\vert)b(t),
  \end{eqnarray*}
which implies that
  \begin{eqnarray*}
  G(t,x)\leq\frac{\vert x\vert^\mu}{M^\mu}\max_{\vert x\vert\leq M}a(\vert x\vert)b(t).
  \end{eqnarray*}
So, one gets
  \begin{eqnarray}\label{Oe24}
  G(t,x)\geq\frac{\vert x\vert^\mu}{M^\mu}\min_{\vert x\vert\leq M}a(\vert x\vert) b(t),
  \end{eqnarray}
Therefore, for any $u\in W_{\Delta,a^+}^{\alpha,p}\setminus\{0\}$, $\xi\in\mathbb{R}^+$, it follows from (\ref{37}), (\ref{Oe1}), (\ref{Oe24}) and $\mu>p^2$ that
\begin{align}\label{Oe25}
\varphi(\xi u)
=&\frac{1}{\varrho p^2}\left(\beta+\varrho\int_{J^0}\vert{^{\mathbb{T}}_aD_t^\alpha} (\xi u)(t)\vert^p\Delta t\right)^p
-\int_{J^0} \lambda (t)G(t,\xi u(t))\Delta t-\frac{\beta^p}{\varrho p^2}\nonumber\\
=&\frac{1}{\varrho p^2}(\beta+\varrho\|\xi u\|^p)^p
-\int_{J^0} \lambda (t)G(t,\xi u(t))\Delta t-\frac{\beta^p}{\varrho p^2}\nonumber\\
\leq&\frac{1}{\varrho p^2}(\beta+\varrho\|\xi u\|^p)^p-\frac{\lambda_0\min\limits_{\vert\xi x\vert\leq M}a(\vert\xi x\vert)}{M^\mu}\vert\xi\vert^\mu\| u\|_\infty^\mu\int_{J^0}b(t)\Delta t-\frac{\beta^p}{\varrho p^2}\nonumber\\
=&\frac{1}{\varrho p^2}(\beta+\varrho\|\xi u\|^p)^p-\frac{\lambda_0\min\limits_{\vert\xi x\vert\leq M}a(\vert\xi x\vert)\|b\|_{L_\Delta^1}\| u\|_\infty^\mu}{M^\mu}\vert\xi\vert^\mu-\frac{\beta^p}{\varrho p^2}\nonumber\\
\rightarrow&-\infty,\quad \mathrm{as}\,\,\xi\rightarrow\infty.
\end{align}
Therefore, taking $\xi_0$ large enough and letting $e=\xi_0u$, we have $\varphi(e)\leq0$. As a consequence, $\varphi$ also satisfies the $(\mathbf{A_2})$ condition in Lemma \ref{OL1}.
\end{itemize}
As a result, we get a critical point $u^*$ of $\varphi$ satisfying $\varphi(u^*)\geq\sigma>0$, and so $u^*$ is a nontrivial solution of KFBVP$_{\mathbb{T}}$  (\ref{O1}). All in all, Theorem \ref{OT1} is proved by $\mathbf{Step\,1}$-$\mathbf{Step\,3}$.
\end{proof}

\begin{theorem}\label{OT2}
Let $\alpha\in\left(\frac{1}{p},1\right]$, suppose that $G$ satisfies $(\mathbf{G_1})$ and the following conditions:
\begin{itemize}
  \item [$(G_4)$]
  There are a constant $1<r_1<p^2$ and a function $d\in L^1_\Delta(J,\mathbb{R}^+)$ such that
  \begin{eqnarray*}
  \vert\nabla G(t,x)\vert\leq r_1d(t)\vert x\vert^{r_1-1},\quad \forall (t,x)\in J\times\mathbb{R}.
  \end{eqnarray*}
  \item [$(G_5)$]
  There is an open interval ${\mathbb{I}}\subset J$ and three constants $\eta,\delta>0$, $1<r_2<p^2$ such that
  \begin{eqnarray*}
  G(t,x)\geq \eta\vert x\vert^{r_2},\quad \forall (t,x)\in {\mathbb{I}}_{\mathbb{T}}\times\vert-\delta,\delta\vert.
  \end{eqnarray*}
\end{itemize}
Then, KFBVP$_{\mathbb{T}}$ (\ref{O1}) has at least one nontrivial weak solution.
\end{theorem}

\begin{proof}
It is obvious that $\varphi(0)=0$, $\varphi\in C^1(W_{\Delta,a^+}^{\alpha,p},\mathbb{R})$, where $W_{\Delta,a^+}^{\alpha,p}$ is a real Banach space from Theorem \ref{3.6}, next, we will take the help of Lemma \ref{OL2} to demonstrate
\begin{itemize}
  \item [($1$).]
  $\varphi$ is bounded from below in $W_{\Delta,a^+}^{\alpha,p}$, which can be explained by the following reasons:

  Taking account of $(\mathbf{G_4})$, (\ref{34}) and (\ref{37}),  we get
\begin{align}\label{Oe26}
\varphi(u)
=&\frac{1}{\varrho p^2}\left(\beta+\varrho\int_{J^0}\vert{^{\mathbb{T}}_aD_t^\alpha} u(t)\vert^p\Delta t\right)^p
-\int_{J^0} \lambda (t)G(t,u(t))\Delta t-\frac{\beta^p}{\varrho p^2}\nonumber\\
=&\frac{1}{\varrho p^2}(\beta+\varrho\|u\|^p)^p
-\int_{J^0} \lambda (t)G(t,u(t))\Delta t-\frac{\beta^p}{\varrho p^2}\nonumber\\
\geq&\frac{1}{\varrho p^2}(\beta+\varrho\|u\|^p)^p
-\lambda^0\int_{J^0} d(t)\vert u(t)\vert^{r_1}\Delta t-\frac{\beta^p}{\varrho p^2}\nonumber\\
\geq&\frac{1}{\varrho p^2}(\beta+\varrho\|u\|^p)^p
-\lambda^0\|d\|_{L_\Delta^1}\|u\|_\infty^{r_1}-\frac{\beta^p}{\varrho p^2}\nonumber\\
\geq&\frac{1}{\varrho p^2}(\beta+\varrho\|u\|^p)^p
-\frac{\lambda^0\|d\|_{L_\Delta^1}b^{r_1(\alpha-\frac{1}{p})}}{\Gamma^{r_1}(\alpha)((\alpha-1)q+1)^{\frac{r_1}{q}}}
\|u\|^{r_1}-\frac{\beta^p}{\varrho p^2}.
\end{align}
Since $1<r_1<p^2$, (\ref{Oe26}) yields $\varphi(u)\rightarrow\infty$ as $\|u\|\rightarrow\infty$. Consequently, $\varphi$ is bounded from below in $W_{\Delta,a^+}^{\alpha,p}$.

  \item [($2$).]
  $\varphi$ satisfies the $(PS)$ condition in $W_{\Delta,a^+}^{\alpha,p}$. The argument is as follows:

Let $\{u_k\}\subset W_{\Delta,a^+}^{\alpha,p}$ be a sequence such that (\ref{Oe4}) holds. So, together with proof by contradiction and (\ref{Oe26}), it is easily for us to see that $\{u_k\}\subset W_{\Delta,a^+}^{\alpha,p}$ is bounded in $W_{\Delta,a^+}^{\alpha,p}$. The remainder of proof is similar to the proof of Step 1 in Proof of Theorem \ref{OT1}. We omit the details.

Consequently, combining with Lemma \ref{OL2}, $\mathbf{(1)}$ and $\mathbf{(2)}$ in Proof of Theorem \ref{OT2}, one gets $c=\inf\limits_{W_{\Delta,a^+}^{\alpha,p}}\varphi(u)$ which is a critical value of $\varphi$. In other words, there is a critical point $u^*\in W_{\Delta,a^+}^{\alpha,p}$ such that $\varphi(u^*)=c$.

  \item [($3$).]
  $u^*\neq0$, for the following reasons:

Let $u_0\in (W^{1,2}_{\Delta,T}(\mathbb{I},\mathbb{R})\cap W_{\Delta,a^+}^{\alpha,p})\setminus\{0\}$ (\cite{2}) and $\|u_0\|_\infty=1$, it follows from (\ref{Oe1}), (\ref{37}), $\mathbf{(G_5)}$ and (\ref{33}) that
\begin{align}\label{Oe27}
\varphi(\varsigma u_0)
=&\frac{1}{\varrho p^2}\left(\beta+\varrho\int_{J^0}\vert{^{\mathbb{T}}_aD_t^\alpha} (\varsigma u_0)(t)\vert^p\Delta t\right)^p
-\int_{J^0} \lambda (t)G(t,\varsigma u_0(t))\Delta t-\frac{\beta^p}{\varrho p^2}\nonumber\\
\leq&\frac{1}{\varrho p^2}(\beta+\varrho\|\varsigma u_0\|^p)^p
-\int_{\mathbb{I}} \lambda (t)G(t,\varsigma u_0(t))\Delta t-\frac{\beta^p}{\varrho p^2}\nonumber\\
\leq&\frac{1}{\varrho p^2}(\beta+\varrho\|\varsigma u_0\|^p)^p
-\lambda_0\eta\varsigma^{r_2}\int_{\mathbb{I}} \vert u_0(t)\vert^{r_2}\Delta t-\frac{\beta^p}{\varrho p^2},\quad 0<s\leq\delta.
\end{align}
Because of $1<r_2<p^2$, (\ref{Oe27}) implies $\varphi(\varsigma u_0)<0$ for $s>0$ small enough. Therefore, $u^*\neq0$.
\end{itemize}
All in all, $u^*\in W_{\Delta,a^+}^{\alpha,p}$ is a nontrivial critical point of $\varphi$, and consequently, $u^*$ is a nontrivial solution of KFBVP$_{\mathbb{T}}$ (\ref{O1}). Hence, we complete the proof of Theorem \ref{OT2}.
\end{proof}

\begin{theorem}\label{OT3}
Let $\alpha\in\left(\frac{1}{p},1\right]$, suppose that $G$ satisfies $(\mathbf{G_1})$, $(\mathbf{G_4})$, $(\mathbf{G_5})$ and the following conditions:
\begin{itemize}
  \item [$(G_6)$]
  There are a constant $1<r_1<p^2$ and a function $d\in L^1_\Delta(J,\mathbb{R}^+)$ such that
  \begin{eqnarray*}
  \nabla G(t,x)=\nabla G(t,-x),\quad \forall (t,x)\in J\times\mathbb{R}.
  \end{eqnarray*}
  \end{itemize}
Then, KFBVP$_{\mathbb{T}}$ (\ref{O1}) possesses infinitely many nontrivial weak solutions.
\end{theorem}

\begin{proof}
Lemma \ref{OL3} is a powerful tool for us to clarify our conclusion.

Considering $\mathbf{(1)}$ and $\mathbf{(2)}$ in Proof of Theorem \ref{OT2}, we see that $\varphi\in C^1(W_{\Delta,a^+}^{\alpha,p},\mathbb{R})$ is bounded from below and satisfies the $(PS)$ condition. Furthermore, it follows from (\ref{Oe1}) and $(\mathbf{G_6})$ that $\varphi$ is even and $\varphi(0)=0$.

Fixing $n\in\mathbb{N}$, we take $n$ disjoint open intervals $\mathbb{I}_i$ such that
$ \mathop{\cup}\limits_{i=1}^n \mathbb{I}_i\subset\mathbb{I}$.

Let $u_i\in (W^{1,2}_{\Delta,T}(\mathbb{I}_i,\mathbb{R})\cap W_{\Delta,a^+}^{\alpha,p})\setminus\{0\}$ and $\|u_i\|=1$, and
  \begin{eqnarray}\label{Oe28}
  W_n=\mathrm{span}\{u_1,u_2,\cdots,u_n\},\nonumber\\
  D_n=\{u\in W_n\vert\|u\|=1\}.
  \end{eqnarray}
For $u\in W_n$, there are $\kappa_i\in\mathbb{R}$ such that
  \begin{eqnarray}\label{Oe29}
  u(t)=\sum_{i=1}^n\kappa_iu_i(t),\quad\forall\,t\in J.
  \end{eqnarray}
Consequently, one obtains
\begin{align}\label{Oe30}
\|u\|^p
=&\int_{J^0}\vert{^{\mathbb{T}}_aD_t^\alpha} u(t)\vert^p\Delta t\nonumber\\
=&\sum_{i=1}^n\vert\kappa_i\vert^p\int_{J^0}\vert{^{\mathbb{T}}_aD_t^\alpha} u_i(t)\vert^p\Delta t\nonumber\\
=&\sum_{i=1}^n\vert\kappa_i\vert^p\|u_i\|^p\nonumber\\
=&\sum_{i=1}^n\vert\kappa_i\vert^p,\quad \forall \,u\in W_n.
\end{align}
In consideration of (\ref{34}), (\ref{37}), (\ref{Oe1}) and ($\mathbf{G_5}$), for $0<\iota\leq\frac{\delta}{\frac{b^{\alpha-\frac{1}{p}}}{\Gamma(\alpha)((\alpha-1)q+1)
^{\frac{1}{q}}}\max\limits_{i=1,2,\cdots,n}\vert\kappa_i\vert}$ and $u\in D_n$, wet get
\begin{align}\label{Oe31}
\varphi(\iota u)
=&\frac{1}{\varrho p^2}\left(\beta+\varrho\int_{J^0}\vert{^{\mathbb{T}}_aD_t^\alpha} (\iota u)(t)\vert^p\Delta t\right)^p
-\int_{J^0} \lambda (t)G(t,\iota u(t))\Delta t-\frac{\beta^p}{\varrho p^2}\nonumber\\
=&\frac{1}{\varrho p^2}(\beta+\varrho\|\iota u\|^p)^p
-\sum_{i=1}^n\int_{\mathbb{I}_i} \lambda (t)G(t,\iota \kappa_iu_i(t))\Delta t-\frac{\beta^p}{\varrho p^2}\nonumber\\
\leq&\frac{1}{\varrho p^2}(\beta+\varrho\|\iota u_0\|^p)^p
-\lambda_0\eta\iota^{r_2}\sum_{i=1}^n\kappa_i^{r_2}\int_{\mathbb{I}_i} \vert u_i(t)\vert^{r_2}\Delta t-\frac{\beta^p}{\varrho p^2},
\end{align}
Since $1<r_2<p^2$, together with (\ref{Oe31}), there are two positive constants $\epsilon$, $\sigma$ such that
  \begin{eqnarray}\label{Oe32}
  \varphi(\sigma u)<-\epsilon,\quad\forall\,u\in D_n.
  \end{eqnarray}
Set
  \begin{eqnarray}\label{Oe33}
  &D_n^\sigma=\{\sigma u\vert u\in D_n\},\nonumber\\
  &\Pi=\left\{(\kappa_1,\kappa_2,\cdots,\kappa_n)\in\mathbb{R}^n\bigg\vert\sum\limits_{i=1}^n|\kappa_i|^p<\sigma^p\right\}.
  \end{eqnarray}
Hence, in view of (\ref{Oe32}), one has
  \begin{eqnarray}\label{Oe34}
  \varphi(u)<-\epsilon,\quad\forall\,u\in D_n^\sigma.
  \end{eqnarray}
Together with the fact of $\varphi$ is even and $\varphi(0)=0$, we obtain
  \begin{eqnarray}\label{Oe35}
  D_n^\sigma\subset\varphi^{-\epsilon}\subset\Xi.
  \end{eqnarray}
By (\ref{Oe30}), we see that the mapping $(\kappa_1,\kappa_2,\cdots,\kappa_n)\rightarrow\sum\limits_{i=1}^n\kappa_iu_i(t)$ from $\partial\Pi$ to $D_n^\sigma$ is odd and homeomorphic. As a result, combining with Propositions 7.5 and 7.7 in \cite{O3}, one gets
  \begin{eqnarray}\label{Oe36}
  \gamma(\varphi^{-\epsilon})\geq\gamma(D_n^\sigma)=n.
  \end{eqnarray}
Hence, $\varphi^{-\epsilon}\in\Xi_n$ and so $\Xi_n\neq0$. Let
  \begin{eqnarray}\label{Oe37}
  c_n=\inf_{A\in\Xi_n}\sup_{u\in A}\varphi(u).
  \end{eqnarray}
It follows from $\varphi$ is bounded from below that $-\infty<c_n\leq-\epsilon<0$. In other words, for any $n\in\mathbb{N}$, $c_n$ is a real negative number.

Consequently, considering Lemma \ref{OL3}, we see that $\varphi$ admits infinitely many nontrivial critical points, and so,  KFBVP$_{\mathbb{T}}$ (\ref{O1}) possesses infinitely many nontrivial weak solutions. The proof of Theorem \ref{OT3} is complete.
\end{proof}

\section{Conclusions}
\setcounter{equation}{0}
\indent

In this paper, a class of fractional Sobolev spaces on time scales is introduced through a new definition of the fractional derivative of Riemann $- $Liouville on time scales, and some basic properties of them are obtained. As an application, we study a class of Kirchhoff type fractional $p $- Laplace boundary value problems on time scales. The existence and multiplicity of nontrivial weak solutions are obtained by using mountain path theorem and genus properties. The results and methods of this paper can also be used to study the solvability of other boundary value problems on time scales. Nowadays, the concepts of fractional derivative on time scales in different senses are constantly put forward. Therefore, it is our future direction to study the theory and application of fractional Sobolev spaces on time scales introduced by fractional derivatives in other senses on time scales.

\end{document}